\documentclass{article}
\usepackage{amsmath, amssymb, graphics}

\usepackage{amsthm}
\usepackage[dvipdfmx]{graphicx}
\usepackage[dvipdfmx]{color}
\usepackage{amscd} 
\usepackage{tikz-cd}
\usepackage{cases}

\newtheorem{thm}{Theorem}[section]
\newtheorem{df}{Definition}[section]
\newtheorem{rem}{Remark}[section]

\newtheorem{prp}{Proposition}[section]
\newtheorem{lem}{Lemma}[section]
\newtheorem{cor}{Corollary}[section]

\newcommand{\mathsym}[1]{{}}
\newcommand{\unicode}[1]{{}}

\makeatletter
    
    \@addtoreset{equation}{section}
  \makeatother
  
\begin{document}

\title{Boolean polynomial expression of the discrete average on the space of rhythms}
\author{Fumio HAZAMA\\Tokyo Denki University\\Hatoyama, Hiki-Gun, Saitama JAPAN\\
e-mail address:hazama@mail.dendai.ac.jp\\Phone number: (81)49-296-5267}
\date{\today}
\maketitle
\thispagestyle{empty}

\begin{abstract}
An infinite family of Boolean polynomials which correspond to the discrete average maps, defined in [2], is constructed.  Each member $Bav_N^0$ of the family is found to be balanced. A recursive formula, which computes $Bav_{N+1}^0$ from $Bav_N^0$, is provided.
\end{abstract}

\noindent
$\mathbf{Keywords}$: Boolean polynomial, space of rhythms, discrete average, Boolean average, balanced polynomial

\newpage
\tableofcontents

\newpage
\section{Introduction}
The main purpose of this paper is to construct a Boolean counterpart of the discrete average map $Rav$, investigated in  [2], and to study its algebraic and combinatorial properties. The map $Rav$ is defined on a space $\mathbf{R}_N^n$ of rhythms with $N$ beats and with $n$ onsets, and it is shown that an arbitrary rhythm in $\mathbf{R}_N^n$ can be made smooth by a finite number of iterations of $Rav$. An observation which motivates our study in the present paper is that a certain quotient space $\overline{\mathbf{R}}_N^n$ of $\mathbf{R}_N=\cup_{n=0}^N\mathbf{R}_N^n$ is found to be naturally isomorphic to the Boolean space $\mathbf{B}_N=\mathbf{F}_2^N$, and that the discrete average descends to a self map on $\overline{\mathbf{R}}_N^n$. It follows that there should be a Boolean counterpart $Bav:\mathbf{B}_N\rightarrow\mathbf{B}_N$, which might shed light on the true nature of $Rav$ from the Boolean viewpoint. We do \underline{not} assume the reader's acquaintance with any theory of music. We will encounter some unexpected algebraic phenomena which might be entitled to be investigated independently of its rhythm-theoretical origin. \\
As a self map on $\mathbf{B}_N$, the Boolean average $Bav$ consists of $N$ Boolean functions. We show, however, that only one of them, called $Bav_N^0$, determines the other $N-1$ functions completely. Furthermore we derive a formula which expresses $Bav_N^0$ as a sum of a simple-looking products (Theorem 5.2), and obtain a recurrence formula which computes $Bav_N^0$ from $Bav_{N-1}^0$ (Theorem 5.3). As a consequence, we see that the Boolean polynomial $Bav_N^0$ is {\it balanced}. According to [4, Chapter 3.3], ``Boolean functions in cryptographic applications almost always need to be balanced, ... .'' Our polynomials $Bav_N^0$ for $N\geq 3$ provide us unexpectedly with an infinite family of balanced ones.   \\
The plan of this paper is as follows. In Section one, we recall the definition of the space $\mathbf{R}_N^n$ of rhythms and of the discrete average $Rav$ on $\mathbf{R}_N^n$. The space $\mathbf{R}_N^n$ is defined to be a subset of the self-product $\mathbf{Z}_N^n$ of $\mathbf{Z}_N=\{0,1,\cdots, N-1\}$ equipped with the addition and subtraction operations inherited from those on the abelian group $\mathbf{Z}/N\mathbf{Z}$. Thus rhythms are originally inhabitants of the ``modulo-$N$-world.'' Since our investigation in this article is based on the fact that $Rav$ preserves the space $\mathbf{R}_N^n$, we give a simplified proof of this fact here. In Section two, we construct a bridge from this modulo-$N$-world to the Boolean world by introducing a certain equivalence relation on $\mathbf{R}_N$, which is compatible with the map $Rav$. The quotient space $\overline{\mathbf{R}}_N^n$ by this relation is found to be naturally isomorphic to the Boolean space $\mathbf{B}_N$, and hence we obtain Boolean counterpart $Bav_N$, called {\it Boolean average}, on $\mathbf{B}_N$. Section three is devoted to deduce several functorial properties of the Boolean average, which enable us to show that only one of the coordinate $Bav_N^0$ of the self map $Bav_N$ is enough to determine the other coordinates. In Section four we recall a standard way to translate a given propositional function to a Boolean polynomial. By this method one can express the $N$ coordinates of $Bav$ as Boolean polynomials. However several examples of the polynomial expressions of $Bav_N^0$ for small $N$ are listed in Table 4.3, they do not give any clue for our attempt to express the polynomials in a unified way. This obstacle leads us to change the basic set of indices $\mathbf{Z}_N$ of Boolean vectors to the set $\mathbf{Z}_{N,\pm}$ of the least absolute remainders modulo $N$. Thanks to the transition to $\mathbf{Z}_{N,\pm}$ we obtain in Section five a simple formula which expresses the polynomial $Bav_N^0$. In order to formulate the result we introduce the notion of a "parental pair of zero", and that of an "ancestor of zero" for which $Bav_N^0$ takes value 1. By determining the structure of the set of ancestors of zero, we find an explicit form of $Bav_N^0$ as well as some recurrence formulas. Furthermore we derive from the structure of the set of the ancestors of zero that $Bav_N^0$ is always balanced for any $N\geq 3$.

\section{The discrete average on the space of rhythms $\mathbf{R}_N^n$}
In this section we recall some definitions in our previous paper [2], and reformulate them in more geometric language. Throughout the present article we fix an integer $N$ which is assumed to be greater than or equal to 3.

\subsection{$\mathbf{Z}_N$-interval and $S^1$-interval}
Let $\mathbf{Z}_N=\{0,1,\cdots, N-1\}$. We denote by $\langle n\rangle_N$ the least nonnegative remainder of an integer $n$ divided by $N$. We define the addition "$+_N$" and the difference "$-_N$" on $\mathbf{Z}_N$ by the rules
\begin{eqnarray*}
a+_Nb&=&\langle a+b\rangle_N,\\
a-_Nb&=&\langle a-b\rangle_N,
\end{eqnarray*}
for any $a,b\in\mathbf{Z}_N$. For any pair $a,b\in\mathbf{Z}_N$, we define the $\mathbf{Z}_N$-{\it interval } $[a,b]_{\mathbf{Z}_N}$ by
\begin{eqnarray*}
[a,b]_{\mathbf{Z}_N}=\left\{
\begin{array}{ll}
\{a,a+1,\cdots,b\}, & \mbox{if $a\leq b$},\\
\{a,a+1,\cdots,N-1,0,1,\cdots,b\}, & \mbox{if $a>b$},
\end{array}
\right.
\end{eqnarray*}
and define half-closed intervals by
\begin{eqnarray*}
[a,b)_{\mathbf{Z}_N}&=&[a,b]_{\mathbf{Z}_N}\setminus\{b\},\\ 
(a,b]_{\mathbf{Z}_N}&=&[a,b]_{\mathbf{Z}_N}\setminus\{a\}.
\end{eqnarray*}
Therefore the number $|[a,b)_{\mathbf{Z}_N}|$ of the elements in $[a,b)_{\mathbf{Z}_N}$ is given by 
\begin{eqnarray}
|[a,b)_{\mathbf{Z}_N}|=b-_Na.
\end{eqnarray}
We call the number $b-_Na$ the $\mathbf{Z}_N$-{\it distance from} $a$ {\it to} $b$, and denote it by $d_{\mathbf{Z}_N}(a,b)$. Note that, when $a\neq b$, the distance satisfies the equality
\begin{eqnarray*}
d_{\mathbf{Z}_N}(a,b)=N-d_{\mathbf{Z}_N}(b,a),
\end{eqnarray*}
hence it is \underline{not} necessarily symmetric. These two notions, $\mathbf{Z}_N$-interval and $\mathbf{Z}_N$ distance, are translated into geometric ones through a character on the additive group $\mathbf{R}$ as follows. Let $S^1=\{z\in\mathbf{C}:|z|=1\}\subset \mathbf{C}$ be the unit circle in the complex plane, and let $\chi_N:\mathbf{R}\rightarrow \mathbf{C}^*=\mathbf{C}\setminus\{0\}$ denote the additive character defined by $\chi_N(t)=\exp(2\pi it/n)$ for any $t\in\mathbf{R}$. Note that $\chi_N(\mathbf{R})=S^1$ and that $\chi_N(\mathbf{Z}_N)=\mu_N$, where $\mu_N$ denotes the group of the $N$-th roots of unity. For any pair $(z,w)\in S^1\times S^1$, we define the $S^1$-{\it interval} $[z,w]_{S^1}\subset S^1$ to be the counterclockwise closed circular arc from $z$ to $w$, and put 
\begin{eqnarray*}
[z,w)_{S^1}&=&[z,w]_{S^1}\setminus \{w\},\\
(z,w]_{S^1}&=&[z,w]_{S^1}\setminus \{z\}.
\end{eqnarray*}
 Let $d_{S^1}(z,w)$ denote the length of the arc $[z,w]_{S^1}$. We call $d_{S^1}(z,w)$ the $S^1$-{\it distance from }$z$ {\it to} $w$. Note that we have
\begin{eqnarray*}
d_{S^1}(z,w)=2\pi-d_{S^1}(w,z),
\end{eqnarray*}
for any distinct pair of elements $z,w\in S^1$. The $\mathbf{Z}_N$-distance and the $S^1$-distance are related by the equation
\begin{eqnarray}
d_{S^1}(\chi_N(a),\chi_N(b))=\frac{2\pi}{N}d_{\mathbf{Z}_N}(a,b),
\end{eqnarray}
for any $(a,b)\in\mathbf{Z}_N\times\mathbf{Z}_N$. For any pair $(z,w)\in \mu_N\times \mu_N$, we define the $\mu_N$-{\it interval} $[z,w]_{\mu_N}$ by 
\begin{eqnarray*}
[z,w]_{\mu_N}&=&[z,w]_{S^1}\cap \mu_N,
\end{eqnarray*}
and let  
\begin{eqnarray*}
[z,w)_{\mu_N}&=&[z,w]_{\mu_N}\setminus \{w\},\\
(z,w]_{\mu_N}&=&[z,w]_{\mu_N}\setminus \{z\}.
\end{eqnarray*}
Note that, for any $(a,b)\in\mathbf{Z}_N\times \mathbf{Z}_N$, the map $\chi_N$ defines a bijection from the $\mathbf{Z}_N$-interval $[a,b]_{\mathbf{Z}_N}$ to the $\mu_N$-interval $[\chi_N(a),\chi_N(b)]_{\mu_N}$.

\subsection{$\mathbf{Z}_N$-average and $\mu_N$-average}
First we define a map $av_{\mathbf{Z}_N}:\mathbf{Z}_N\times\mathbf{Z}_N\rightarrow \mathbf{Z}_N$ by the rule
\begin{eqnarray}
av_{\mathbf{Z}_N}(a,b)=a+_N\left\lfloor\frac{b-_Na}{2}\right\rfloor
\end{eqnarray}
for any $(a,b)\in \mathbf{Z}_N\times\mathbf{Z}_N$. However we use the notation ``$rav$'' in [2] to express this map $av_{\mathbf{Z}_N}$, we rename it in this paper in order to contrast it with a corresponding average on $\mu_N$. The latter is defined as follows. For any $(z,w)\in S^1\times S^1$, we denote the midpoint of the circular arc $[z,w]_{S^1}$ by $av_{S^1}(z,w)$, and call it $S^1$-{\it average} of $z$ and $w$. Furthermore a kind of {\it rounding-off} operator $\lfloor\cdot\rfloor_{\mu_N}$ on $S^1$ is introduced as follows. Note that we have the equality
\begin{eqnarray*}
S^1=\bigsqcup_{k=0}^{N-1}[\zeta_N^k,\zeta_N^{k+1})_{S^1},
\end{eqnarray*}
where $\zeta_N=\chi_N(1)$. Hence for any $z\in S^1$, there exists a unique $k\in \mathbf{Z}_N$ such that $z\in [\zeta_N^k,\zeta_N^{k+1})_{S^1}$. We put $\lfloor z\rfloor_{\mu_N}=\zeta_N^k$ employing this $k$. The $\mu_N$-{\it average} $av_{\mu_N}(z,w)$ of $(z,w)\in\mu_N\times\mu_N$ is define by 
\begin{eqnarray*}
av_{\mu_N}(z,w)=\lfloor av_{S^1}(z,w)\rfloor_{\mu_N}.
\end{eqnarray*}
The following proposition shows that the $\mathbf{Z}_N$-average corresponds to the $\mu_N$-average through the bijection $\chi_N|_{\mathbf{Z}_N}$.

\begin{prp}
For any $(a,b)\in\mathbf{Z}_N\times\mathbf{Z}_N$, we have
\begin{eqnarray}
\chi_N(av_{\mathbf{Z}_N}(a,b))=av_{\mu_N}(\chi_N(a),\chi_N(b)),
\end{eqnarray}
namely, the diagram
\begin{eqnarray*}
  \begin{CD}
     \mathbf{Z}_N\times\mathbf{Z}_N @>{\chi_N\times\chi_N}>> \mu_N\times\mu_N \\
  @V{av_{\mathbf{Z}_N}}VV @V{av_{\mu_N}}VV \\
      \mathbf{Z}_N @>{\chi_N}>> \mu_N \\
  \end{CD}
\end{eqnarray*}
commutes.
\end{prp}

\begin{proof}
For the proof, we employ the following lemma.

\begin{lem}
For any $(a,b)\in\mathbf{Z}_N\times \mathbf{Z}_N$, we have
\begin{eqnarray}
av_{\mathbf{Z}_N}(a,b)=
\left\{
\begin{array}{ll}
\left\lfloor\frac{a+b}{2}\right\rfloor, & \mbox{if }a\leq b,\\
\langle\left\lfloor\frac{a+b+N}{2}\right\rfloor\rangle_N, & \mbox{if }a>b.\\
\end{array}
\right.
\end{eqnarray}
\end{lem}

\noindent
{\it Proof of Lemma} 1.1. Suppose that $a\leq b$. Then we have
\begin{eqnarray*}
\chi_N(av_{\mathbf{Z}_N}(a,b))&=&\chi_N(a+_N\left\lfloor\frac{b-_Na}{2}\right\rfloor)\\
&=&\chi_N(a)\chi_N(\left\lfloor\frac{b-a}{2}\right\rfloor)\\
&=&\chi_N(a+\left\lfloor\frac{b-a}{2}\right\rfloor)\\
&=&\chi_N(\left\lfloor\frac{a+b}{2}\right\rfloor).
\end{eqnarray*}
Since $0\leq a\leq b \leq N-1$, we have $0\leq \left\lfloor\frac{a+b}{2}\right\rfloor\leq N-1$, and hence the assertion (1.5) holds by the bijectivity of $\chi_N|_{\mathbf{Z}_N}$. When $a>b$, we have $b-_Na=b-a+N$. Hence we have
\begin{eqnarray*}
\chi_N(av_{\mathbf{Z}_N}(a,b))&=&\chi_N(a+_N\left\lfloor\frac{b-_Na}{2}\right\rfloor)\\
&=&\chi_N(a)\chi_N(\left\lfloor\frac{b-a+N}{2}\right\rfloor)\\
&=&\chi_N(a+\left\lfloor\frac{b-a+N}{2}\right\rfloor)\\
&=&\chi_N(\left\lfloor\frac{a+b+N}{2}\right\rfloor).
\end{eqnarray*}
This time we need to take the remainder modulo $N$ of $\left\lfloor\frac{a+b+N}{2}\right\rfloor$, which need not be smaller than $N$. This finishes the proof of Lemma 1.1.\\
With the help of this lemma, we can prove Proposition 1.1. When $a\leq b$, by the definition of $av_{S^1}$, we see that $av_{S^1}(\chi_N(a),\chi_N(b))=\chi_N(\frac{a+b}{2})$, since the usual average $(a+b)/2$ is the midpoint of the interval $[a,b]\subset\mathbf{R}$. Therefore $av_{\mu_N}(\chi_N(a),\chi_N(b))=\chi_N(\lfloor \frac{a+b}{2}\rfloor)$, which is equal to $\chi_N(av_{\mathbf{Z}_N}(a,b))$ by Lemma 1.1.  The case when $a>b$ can be treated similarly also by Lemma 1.1. 
\end{proof}

\noindent
As a consequence, we can prove the following.

\begin{prp}
(1) For any $(z,w)\in\mu_N\times\mu_N$, we have
\begin{eqnarray*}
av_{\mu_N}(z,w)\in [z,w]_{\mu_N}.
\end{eqnarray*}
(2) For any $(a,b)\in\mathbf{Z}_N\times\mathbf{Z}_N$, we have
\begin{eqnarray*}
av_{\mathbf{Z}_N}(a,b)\in [a,b]_{\mathbf{Z}_N}.
\end{eqnarray*}
\end{prp}

\begin{proof}
(1) The midpoint $av_{S^1}(z,w)$ belongs to $[z,w]_{S^1}$ by definition, and hence $av_{\mu_N}(z,w)=\lfloor av_{S^1}(z,w)\rfloor_{\mu_N}\in [z,w]_{\mu_N}$, since the rounding-off operator cannot go beyond $z\in\mu_N$.\\
(2) Let $c=av_{\mathbf{Z}_N}(a,b)$. Then if follows from (1.4) that 
\begin{eqnarray*}
\chi_N(c)=av_{\mu_N}(\chi_N(a),\chi_N(b)),
\end{eqnarray*}
and hence
\begin{eqnarray}
\chi_N(c)\in [\chi_N(a),\chi_N(b)]_{\mu_N}
\end{eqnarray}
by the above (1). Since $\chi_N$ is a bijection from $[a,b]_{\mathbf{Z}_N}$ onto $[\chi_N(a),\chi_N(b)]_{\mu_N}$, the statement (1.6) implies that $c=av_{\mathbf{Z}_N}(a,b)\in [a,b]_{\mathbf{Z}_N}$. This completes the proof.
\end{proof}

\noindent
Noting that the midpoint of $[z,w]_{S^1}$ coincides with $w$ only when $z=w$, we obtain the following.
\begin{cor}
(1) For any pair $(z,w)$ of \underline{distinct} elements of $\mu_N$, we have
\begin{eqnarray*}
av_{\mu_N}(z,w)\in [z,w)_{\mu_N}.
\end{eqnarray*}
(2) For any pair $(a,b)$ of \underline{distinct} elements of $\mathbf{Z}_N$, we have
\begin{eqnarray*}
av_{\mathbf{Z}_N}(a,b)\in [a,b)_{\mathbf{Z}_N}.
\end{eqnarray*}
\end{cor}

\subsection{The space of rhythms and the discrete average transformation}
In this subsection, we recall the definition of the space of rhythms given in [2], and reformulate it by employing the $S^1$-average and the $\mu_N$-average.

\begin{df}
For any positive integer $n$ with $2\leq n \leq N$, an $n$-tuple $\mathbf{a}=(a_0,\cdots,a_{n-1})\in\mathbf{Z}_N^n$ is said to be a {\rm rhythm modulo }$N$ {\rm with} $n$ {\rm onsets}, if  $a_i\neq a_j$ for $i\neq j$, and the following condition holds:
\begin{eqnarray}
\sum_{i=0}^{n-1}(a_{i+1}-_Na_i)=N.
\end{eqnarray}
(Here the summation on the left hand side means the usual addition of integers, and $a_n$ is set to be $a_0$.) We denote by $\mathbf{R}_N^n$ the set of rhythms modulo $N$ with $n$ onsets.
\end{df}

\noindent
This is the original definition given in [2] of the space of rhythms. Note that the condition (1.7) is equivalent to the equality
\begin{eqnarray}
\mathbf{Z}_N=\bigsqcup_{i=0}^{n-1}[a_i,a_{i+_n1})_{\mathbf{Z}_N}.
\end{eqnarray}
Applying the map $\chi_N$ to the both sides of (1.8) we obtain another equivalent condition
\begin{eqnarray}
S^1=\bigsqcup_{i=0}^{n-1}[\chi_N(a_i),\chi_N(a_{i+_n1}))_{S^1},
\end{eqnarray}
which says that the ordered set $(\chi_N(a_0),\cdots,\chi_N(a_{n-1}))$ constitutes a simple convex $n$-gon inscribed in $S^1$. The condition (1.9) clarifies the meaning of the algebraic condition (1.7) in more geometric terms, and will be employed frequently afterwards. \\
In order to make our transition to the Boolean world smooth, we need, however, to define the space $\mathbf{R}_N^1$ of rhythms with 1 onset as well as $\mathbf{R}_N^0$. The following definition fills the gap:

\begin{df}
Any singleton $(a)$ of $a\in\mathbf{Z}_N$ is regarded as a rhythm with 1 onset. The empty set $\phi$ is considered to be the unique rhythm with 0 onset. Accordingly we define
\begin{eqnarray*}
\mathbf{R}_N^1&=&\{(a);a\in\mathbf{Z}_N\}, \\
\mathbf{R}_N^0&=&\{\phi\}.
\end{eqnarray*}
Furthermore we define the total space of rhythms $\mathbf{R}_N$ with $N$ beats by
\begin{eqnarray*}
\mathbf{R}_N=\bigcup_{n=0}^N\mathbf{R}_N^n.
\end{eqnarray*} 
\end{df}

\noindent
The {\it discrete average transformation} on $\mathbf{R}_N$, which is the central topic in [2], is defined as follows.

\begin{df}
For a rhythm $\mathbf{a}=(a_0,\cdots,a_{n-1})\in\mathbf{R}_N^n$ with $2\leq n\leq N$, we define its discrete average transformation $Rav(\mathbf{a})$ by
\begin{eqnarray}
Rav(\mathbf{a})=(av_{\mathbf{Z}_N}(a_0,a_1),av_{\mathbf{Z}_N}(a_1,a_2),\cdots, av_{\mathbf{Z}_N}(a_{n-1},a_0)).
\end{eqnarray}
\end{df}

\noindent
Since we have enlarged the space of rhythms in Definition 1.2, we need to define discrete average on $\mathbf{R}_N^1$ and $\mathbf{R}_N^0$ too. This is accomplished simply as follows:

\begin{df}
For any $(a)\in\mathbf{R}_N^1$, we define
\begin{eqnarray}
Rav((a))=(a),
\end{eqnarray}
and on $\mathbf{R}_N^0$ we put
\begin{eqnarray}
Rav(\phi)=\phi.
\end{eqnarray}
\end{df}

\begin{rem}
It follows from the definition that $av_{\mathbf{Z}_N}(a,a)=0$ holds for any $a\in\mathbf{Z}_N$. Hence our definition (1.11) is natural.
\end{rem}

\noindent
The main purpose of this subsection is to show that the discrete average transformation $Rav$ maps $\mathbf{R}_N$ into itself. This fact is proved in [2] through an algebraic argument. We, however, give below a simpler and more geometric proof.

\begin{prp}
For any $\mathbf{a}\in\mathbf{R}_N^n$, we have $Rav(\mathbf{a})\in\mathbf{R}_N^n$.
\end{prp}

\begin{proof}
We may assume that $n\geq 2$, since the assertion follows directly from Definition 1.4 when $n=0$ or $n=1$. Let $Rav(\mathbf{a})=\mathbf{b}=(b_0,b_1,\cdots,b_{n-1})$. It follows from the assumption that we have
\begin{eqnarray*}
S^1=\bigsqcup_{i=0}^{n-1}[\chi_N(a_i),\chi_N(a_{i+_n1}))_{S^1}.
\end{eqnarray*}
Furthermore it follows from Proposition 1.1 and Corollary 1.1 that
\begin{eqnarray*}
\chi_N(b_i)\in [\chi_N(a_i),\chi_N(a_{i+_n1}))_{\mu_N}.
\end{eqnarray*}
Hence we have 
\begin{eqnarray*}
S^1=\bigsqcup_{i=0}^{n-1}[\chi_N(b_i),\chi_N(b_{i+_n1}))_{S^1}.
\end{eqnarray*}
Therefore the $n$-tuple $\mathbf{b}$ satisfies the condition (1.9), and hence $\mathbf{b}\in \mathbf{R}_N^n$. This completes the proof.
\end{proof}

\section{From modulo-$N$-world to the Boolean world}
The purpose of this section is to construct a Boolean counterpart of the discrete average transformation $Rav$ on $\mathbf{R}_N$. We accomplish this by introducing two intermediate spaces $\overline{\mathbf{R}}_N$ and $\mathbf{I}_N$, which lie between $\mathbf{R}_N$ and our destination, the Boolean space $\mathbf{B}_N=\mathbf{F}_2^N$. We define a certain average map on each of the new spaces, and build three bridges $\pi, s, ItoB$ between each pair of the four spaces, which make the following diagram commutative.

\begin{eqnarray*}
  \begin{CD}
     \mathbf{R}_N @>{\pi}>> \overline{\mathbf{R}}_N @>{s}>> \mathbf{I}_N @>{ItoB}>> \mathbf{B}_N \\
  @V{Rav}VV @V{\overline{R}av}VV @V{Iav}VV @V{Bav}VV \\
     \mathbf{R}_N @>{\pi}>> \overline{\mathbf{R}}_N @>{s}>> \mathbf{I}_N @>{ItoB}>> \mathbf{B}_N\\
  \end{CD}
\end{eqnarray*}

\subsection{Quotient map $\mathbf{R}_N^n\rightarrow\overline{\mathbf{R}}_N^n$ and its section} 
A self map "$rot$" on $\mathbf{R}_N^n$ is introduced in [2]. Let us recall the definition and supply it with additional versions for the case $n=0,1$:

\begin{df}
For any $\mathbf{a}=(a_0,a_1,\cdots,a_{n-1})\in\mathbf{R}_N^n$ with $2\leq n\leq N$, let
\begin{eqnarray}
rot(\mathbf{a})=(a_{n-1},a_0,\cdots,a_{n-2}).
\end{eqnarray}
When $n=0,1$, it is defined to be the identity, namely,
\begin{eqnarray}
rot((a_0))&=&(a_0) \mbox{ for any }(a_0)\in\mathbf{R}_N^1,\\
rot(\phi)&=&\phi.
\end{eqnarray}
\end{df}

\noindent
Note that $rot(\mathbf{a})\in\mathbf{R}_N^n$ for any $\mathbf{a}\in\mathbf{R}_N^n$. We define an equivalence relation on $\mathbf{R}_N^n$ employing this map as follows:

\begin{df}
For any pair $\mathbf{a},\mathbf{b}\in\mathbf{R}_N^n$, they are said to be {\rm equivalent} if there exists an integer $k$ such that 
\begin{eqnarray*}
rot^k(\mathbf{a})=\mathbf{b}.
\end{eqnarray*}
\end{df}

\noindent
Since the map $rot$ is bijective, this defines an equivalence relation on $\mathbf{R}_N^n$, which we denote by "$\sim_{rot}$".

\begin{df}
We denote by $\overline{\mathbf{R}}_N^n$ the quotient of $\mathbf{R}_N^n$ by the equivalence relation $\sim_{rot}$:
\begin{eqnarray*}
\overline{\mathbf{R}}_N^n=\mathbf{R}_N^n/\sim_{rot},
\end{eqnarray*}
and by $\pi:\mathbf{R}_N^n\rightarrow\overline{\mathbf{R}}_N^n$ the natural projection.
\end{df}

\begin{rem}
Our definition of "$rot$" here is actually the inverse of the one used in $[2]$. The reason why we adopt the inverse is that it makes several formulas in this paper look more natural. The equivalence relation "$\sim_{rot}$", however, is actually defined through the action of the group $\mathbf{Z}_n$, generated by $rot$, hence the quotient space $\overline{\mathbf{R}}_N^n$ is exactly the same as the one defined in $[$loc.cit.$]$.
\end{rem}

\begin{rem}
When $n=0,1$, the space $\overline{\mathbf{R}}_N^n$ is nothing other than $\mathbf{R}_N^n$, since $rot$ is the identity on $\mathbf{R}_N^1$ or on $\mathbf{R}_N^0$.
\end{rem}

It follows from Definition 1.3, 1.4, and 2.1 that the two self maps $rot$ and $Rav$ on $\mathbf{R}_N^n$ commute with each other for any $n\geq 0$ :
\begin{eqnarray*}
rot\circ Rav=Rav\circ rot,
\end{eqnarray*}
namely, we have the following commutative diagram:
\begin{eqnarray*}
  \begin{CD}
     \mathbf{R}_N^n @>{rot}>> \mathbf{R}_N^n \\
  @V{Rav}VV @V{Rav}VV \\
     \mathbf{R}_N^n @>{rot}>> \mathbf{R}_N^n \\
  \end{CD}
\end{eqnarray*}
Therefore the map $Rav$ descends to a self map $\overline{R}av$ on $\overline{\mathbf{R}}_N^n$, namely, the equality
\begin{eqnarray}
\pi\circ Rav=\overline{R}av\circ\pi.
\end{eqnarray}
holds. Hence we have the following commutative diagram:
\begin{eqnarray*}
  \begin{CD}
     \mathbf{R}_N^n @>{\pi}>> \overline{\mathbf{R}}_N^n \\
  @V{Rav}VV @V{\overline{R}av}VV \\
     \mathbf{R}_N^n @>{\pi}>> \overline{\mathbf{R}}_N^n \\
  \end{CD}
\end{eqnarray*}

\subsection{Construction of a section of $\pi$}
The purpose of this subsection is to construct a section of the quotient map $\pi:\mathbf{R}_N^n\rightarrow\overline{\mathbf{R}}_N^n$. Here the notion of {\it jumping number}, introduced in [2], plays a crucial role.

\begin{df}
When $n\geq 2$, for any $\mathbf{a}=(a_0,a_1,\cdots, a_{n-1})\in\mathbf{R}_N^n$, there is a unique number $j\in\mathbf{Z}_n$ such that $a_{j-_n1}>a_j$ and that $a_{k-_n1}<a_k$ holds for every other $k\in\mathbf{Z}_n\setminus\{j\}$. The number $j$ is called the {\rm jumping number} of $\mathbf{a}$. When $n=1$, the jumping number of an arbitrary element $(a_0)$ of $\mathbf{R}_N^1$ is defined to be $0$, In any case it is denoted by $j(\mathbf{a})$.
\end{df}

\begin{rem}
Our definition of the jumping number increases the corresponding number defined in $[2]$ by one. One of the main reasons why we adopt this version is that we want every element in the basic set $\mathbf{I}_N^n$, defined in Definition 2.5 later, of increasing sequences to acquire the jumping number $0$. Since we derive several results in this paper by choosing from $\mathbf{I}_N^n$ a representative of the class in $\overline{\mathbf{R}}_N^n$, this convention makes our description simpler.
\end{rem}

We can characterize the jumping number in geometric and intuitive terms by employing the half-closed intervals.

\begin{prp}
Assume that $n\geq 2$. For any $\mathbf{a}=(a_0,a_1,\cdots, a_{n-1})\in\mathbf{R}_N^n$, the jumping number $j(\mathbf{a})$ is equal to the unique number $i\in \mathbf{Z}_n$ such that the half-open interval $(\chi_N(a_{i-_n1}),\chi_N(a_i)]_{S^1}$ contains $1\in S^1$.
\end{prp}

\begin{proof}
Replacing the left-closed and right-open intervals in the condition (1.9) by the left-open and right-closed intervals, we obtain an equivalent condition
\begin{eqnarray*}
S^1=\bigsqcup_{i=0}^{n-1}(\chi_N(a_{i-_n1}),\chi_N(a_{i})]_{S^1}.
\end{eqnarray*}
Hence there exists a unique member $i\in \mathbf{Z}_n$ such that 
\begin{eqnarray*}
1\in (\chi_N(a_{i-_n1}),\chi_N(a_i)]_{S^1}.
\end{eqnarray*}
This condition is equivalent to the condition that 
\begin{eqnarray*}
0\in (a_{i-_n1},a_i]_{\mathbf{Z}_N},
\end{eqnarray*}
which holds if and only if $j(\mathbf{a})=i$. This completes the proof.
\end{proof}

\noindent
By the very definition of the jumping number, we have the following:

\begin{prp}
For any $\mathbf{a}\in\mathbf{R}_N^n$, we have
\begin{eqnarray*}
j(rot(\mathbf{a}))=j(\mathbf{a})+_n1.
\end{eqnarray*}
\end{prp}

\begin{proof}
Let $\mathbf{a}=(a_0,\cdots,a_{n-1})$ and $rot(\mathbf{a})=\mathbf{b}=(b_0,\cdots, b_{n-1})$ so that
\begin{eqnarray}
b_i=a_{i-_n1}
\end{eqnarray}
holds for any $i\in\mathbf{Z}_n$. If we put $j=j(\mathbf{a})$, then we have
\begin{eqnarray*}
0\leq a_{j}<a_{j+_n1}<\cdots <a_{j-_n1}\leq N-1.
\end{eqnarray*}
This series of inequalities is equivalent to
\begin{eqnarray*}
0\leq b_{j+_n1}<b_{j+_n2}<\cdots <b_j\leq N-1
\end{eqnarray*}
by (2.5). Hence $j(\mathbf{b})=j+_n1$. This completes the proof.
\end{proof}

\noindent
As a direct consequence, we obtain the following:

\begin{cor}
For any $x\in\overline{\mathbf{R}}_N^n$, the set of jumping numbers of the elements in the fiber $\pi^{-1}(x)$ coincides with the whole $\mathbf{Z}_n$. 
\end{cor}

As a result, we can choose from each fiber of $\pi$ a unique rhythm with its jumping number equal to zero. For this reason, we introduce the following:

\begin{df}
We denote by $\mathbf{I}_N^n$ the subset of $\mathbf{R}_N^n$ which consists of every element $(a_0,\cdots,a_{n-1})\in\mathbf{R}_N^n$ such that
\begin{eqnarray}
0\leq a_0<\cdots <a_{n-1}\leq N-1,
\end{eqnarray}
namely, we put
\begin{eqnarray*}
\mathbf{I}_N^n=\{\mathbf{a}\in\mathbf{R}_N^n; j(\mathbf{a})=0\}.
\end{eqnarray*}
\end{df}

\begin{rem}
The condition (2.6) says that the sequence $(a_0,\cdots,a_{n-1})$ is \underline{i}ncreasing. We adopt its initial as the name of the space.
\end{rem}

Proposition 2.2 enables one to find the unique element in $\pi^{-1}(x)\cap\mathbf{I}_N^n$ algorhythmically:

\begin{prp}
For any $\mathbf{a}\in\mathbf{R}_N^n$, let
\begin{eqnarray}
pr_{\mathbf{I}}(\mathbf{a})=rot^{n-_nj(\mathbf{a})}(\mathbf{a}).
\end{eqnarray}
Then it defines a contraction map $pr_{\mathbf{I}}:\mathbf{R}_N^n\rightarrow\mathbf{I}_N^n$ such that
\begin{eqnarray}
pr_{\mathbf{I}}\circ\iota_{\mathbf{I}}=id_{\mathbf{I}_N^n},
\end{eqnarray}
where $\iota_{\mathbf{I}}:\mathbf{I}_N^n\rightarrow\mathbf{R}_N^n$ denotes the natural inclusion map. 
\end{prp}

\begin{proof}
It follows form Proposition 2.2 and (2.7) that
\begin{eqnarray*}
j(pr_{\mathbf{I}}(\mathbf{a}))=j(\mathbf{a})+_n(n-_nj(\mathbf{a}))=0.
\end{eqnarray*}
Hence $pr_{\mathbf{I}}(\mathbf{a})\in\mathbf{I}_N^n$. Furthermore, if $\mathbf{a}\in\mathbf{I}_N^n$, then $j(\mathbf{a})=0$, and hence $pr_{\mathbf{I}}(\mathbf{a})=\mathbf{a}$. This shows that the equality (2.8) holds true. This completes the proof.
\end{proof}

\noindent
Thus we obtain a section of the quotient map $\pi$:

\begin{prp}
For any $\mathbf{a}\in\mathbf{R}_N^n$, let
\begin{eqnarray*}
s(\pi(\mathbf{a}))=pr_{\mathbf{I}}(\mathbf{a}).
\end{eqnarray*}
Then $s$ defines a section of the quotient map $\pi:\mathbf{R}_N^n\rightarrow \overline{\mathbf{R}}_N^n$, and its image coincides with the set $\mathbf{I}_N^n$. In particular $\mathbf{I}_N^n$ and $\overline{\mathbf{R}}_N^n$ are in one-to-one correspondence.
\end{prp}

\begin{rem}
When $n=0,1$, we define $s:\overline{\mathbf{R}}_N^n\rightarrow\mathbf{R}_N^n$ to be the identity map. (See Remark 2.2.)
\end{rem}

\subsection{From $\overline{\mathbf{R}}_N^n$ to Boolean world via $\mathbf{I}_N^n$ }
In a standard way, the set $\mathbf{I}_N^n$ is connected directly to the Boolean world. For later use we introduce some notations and make the connection clear:

\begin{df}
Let $\mathbf{B}_N$ denote the $N$-dimensional vector space $\mathbf{F}_2^N$ over the prime field $\mathbf{F}_2=\{0,1\}$. For any vector $\mathbf{v}\in\mathbf{B}_N$, we denote by $supp(\mathbf{v})$ the set of indices of the non-zero entries of $\mathbf{v}$,
\begin{eqnarray*}
supp(\mathbf{v})=\{i\in\mathbf{Z}_N;v_i=1\}\subset\mathbf{Z}_N,
\end{eqnarray*}
and let
\begin{eqnarray*}
\mathbf{B}_N^n=\{\mathbf{v}\in\mathbf{B}_N;|supp(\mathbf{v})|=n\}
\end{eqnarray*}
for any $n$ with $0\leq n\leq N$, so that
\begin{eqnarray*}
\mathbf{B}_N=\bigcup_{n=0}^N\mathbf{B}_N^n.
\end{eqnarray*}
\end{df}

First we relate every element in $\mathbf{R}_N^n$ to $\mathbf{B}_N^n$ by taking its characteristic function:

\begin{df}
With any $\mathbf{a}\in\mathbf{R}_N^n$, we associate a vector 
\begin{eqnarray*}
RtoB(\mathbf{a})=(v_0,\cdots,v_{N-1})\in\mathbf{B}_N
\end{eqnarray*}
by the rule
\begin{eqnarray*}
v_i=
\left\{
\begin{array}{ll}
1, & \mbox{ if $i\in \mathbf{a}$,}\\
0, & \mbox{ if $i\not\in \mathbf{a}$.} \\
\end{array}
\right.
\end{eqnarray*}
This gives us the map
\begin{eqnarray*}
RtoB:\mathbf{R}_N^n\rightarrow \mathbf{B}_N^n.
\end{eqnarray*}
\end{df}

\noindent
Note that the following simple but important equality
\begin{eqnarray}
RtoB\circ rot=RtoB.
\end{eqnarray}
holds, since the map $rot$ does not change the contents of any rhythms. It follows from (2.9) that
\begin{eqnarray}
RtoB\circ pr_{\mathbf{I}}=RtoB,
\end{eqnarray}
since $pr_{\mathbf{I}}$ is a power of $rot$ by the definition (2.7). \\
Next, based on the map $RtoB$, we relate $\mathbf{I}_N^n$ with $\mathbf{B}_N^n$:

\begin{df}
For any $n$-element subset $S$ of $\mathbf{Z}_N$, by arranging its elements in increasing order
\begin{eqnarray*}
0\leq s_0<s_1<\cdots <s_{n-1}\leq N-1,
\end{eqnarray*}
we define
\begin{eqnarray*}
ord(S)=(s_0,\cdots,s_{n-1}).
\end{eqnarray*}
We put
\begin{eqnarray}
ItoB&=&RtoB\circ\iota_{\mathbf{I}}:\mathbf{I}_N^n\rightarrow\mathbf{B}_N^n, \\
BtoI&=&ord\circ supp:\mathbf{B}_N^n\rightarrow\mathbf{I}_N^n.
\end{eqnarray}
\end{df}

\noindent
Since one can check that
\begin{eqnarray}
ItoB\circ BtoI&=&id_{\mathbf{B}_N^n},\\
BtoI\circ ItoB&=&id_{\mathbf{I}_N^n},
\end{eqnarray}
$\mathbf{I}_N^n$ and $\mathbf{B}_N^n$ are in one-to-one correspondence. \\
Finally, we relate the quotient space $\overline{\mathbf{R}}_N^n$ with $\mathbf{B}_N^n$.

\begin{prp}
Let $\overline{R}toB:\overline{\mathbf{R}}_N^n\rightarrow\mathbf{B}_N^n$ be the map defined by
\begin{eqnarray}
\overline{R}toB=ItoB\circ s,
\end{eqnarray}
and let $Bto\overline{R}:\mathbf{B}_N^n\rightarrow\overline{\mathbf{R}}_N^n$ be the map defined by
\begin{eqnarray}
Bto\overline{R}=\pi\circ BtoI.
\end{eqnarray}
These maps are inverse to each other:
\begin{eqnarray}
Bto\overline{R}\circ \overline{R}toB&=&id_{\overline{\mathbf{R}}_N^n}, \\
\overline{R}toB\circ Bto\overline{R}&=&id_{\mathbf{B}_N^n}.
\end{eqnarray}
In particular, the sets $\overline{\mathbf{R}}_N^n$ and $\mathbf{B}_N^n$ are in one-to-one correspondence.
\end{prp}

\begin{proof}
First we compute $Bto\overline{R}\circ \overline{R}toB$ as follows:
\begin{eqnarray*}
Bto\overline{R}\circ \overline{R}toB&=&(\pi\circ BtoI)\circ(ItoB\circ s)\\
&=&\pi\circ s\hspace{15mm}\mbox{ (by (2.14))}\\
&=&id_{\overline{\mathbf{R}}_N^n},
\end{eqnarray*}
hence we have (2.17). For the validity of (2.18), we note that, by Proposition 2.4, the composite map $s\circ\pi:\mathbf{R}_N^n\rightarrow\mathbf{R}_N^n$ becomes the identity when restricted to the subset $\mathbf{I}_N^n$. Hence 
\begin{eqnarray*}
\overline{R}toB\circ Bto\overline{R}&=&(ItoB\circ s)\circ(\pi\circ BtoI)\\
&=&ItoB\circ BtoI\\
&=&id_{\mathbf{B}_N^n},
\end{eqnarray*}
the last equality being a consequence of (2.13). This completes the proof.
\end{proof}

The two maps $RtoB$ and $\overline{R}toB$ are related in a natural way:

\begin{prp}
We have
\begin{eqnarray*}
\overline{R}toB\circ\pi=RtoB.
\end{eqnarray*}
\end{prp}

\begin{proof}
Combining several defining equations, we compute as follows:
\begin{eqnarray*}
\overline{R}toB\circ\pi&=&ItoB\circ s\circ\pi\hspace{10mm}(\mbox{by }(2.15))\\
&=&ItoB\circ pr_{\mathbf{I}}\hspace{12mm}(\mbox{by Proposition }2.4)\\
&=&RtoB\circ\iota_{\mathbf{I}}\circ pr_{\mathbf{I}}\hspace{6mm}(\mbox{by }(2.11))\\
&=&RtoB\circ pr_{\mathbf{I}}\\
&=&RtoB. \hspace{20mm}(\mbox{by }(2.10))
\end{eqnarray*}
This finishes the proof.
\end{proof}

\subsection{Counterpart of $Rav$ on $\mathbf{I}_N^n$}
The discrete average map $Rav:\mathbf{R}_N^n\rightarrow\mathbf{R}_N^n$, to our regret, does {\it not} keep the subspace $\mathbf{I}_N^n\subset\mathbf{R}_N^n$ invariant. For example, for $\mathbf{a}=(2,3,7)\in\mathbf{I}_8^3$, its discrete average is
\begin{eqnarray*}
Rav(2,3,7)=(2,5,0),
\end{eqnarray*}
which does not belong to $\mathbf{I}_8^3$. In order to overcome this inconvenience and to define a self map $Iav$ on $\mathbf{I}_N^n$, we introduce the following notion:

\begin{df}
A rhythm $\mathbf{a}=(a_0,\cdots,a_{n-1})\in\mathbf{I}_N^n$ is said to be {\rm proper}, if
\begin{eqnarray}
N-a_{n-1}>a_0.
\end{eqnarray}
If it is not proper, it is called {\rm improper}. 
\end{df}

\noindent
The following proposition relates the properness of an increasing rhythm with its behavior under the discrete average map.

\begin{prp}
(1) For an increasing rhythm $\mathbf{a}=(a_0,\cdots,a_{n-1})\in\mathbf{I}_N^n$, it is proper if and only if $Rav(\mathbf{a})\in\mathbf{I}_N^n$. \\
(2) If $\mathbf{a}=(a_0,\cdots,a_{n-1})\in\mathbf{I}_N^n$ is improper, then $rot(Rav(\mathbf{a}))\in \mathbf{I}_N^n$.
\end{prp}

\begin{proof}
(1) Let $Rav(\mathbf{a})=\mathbf{b}=(b_0,b_1,\cdots,b_{n-1})$. By Corollary 1.1, (2), we have a series of inequalities
\begin{eqnarray}
0\leq a_0\leq b_0< a_1\leq b_1< a_2\leq \cdots <a_{n-2}\leq b_{n-2}< a_{n-1}\leq N-1.
\end{eqnarray}
Therefore only the behavior of $b_{n-1}=av_{\mathbf{Z}_N}(a_{n-1},a_0)$ determines whether $\mathbf{b}$ belongs to $\mathbf{I}_N^n$ or not. Note that the condition (2.19) is expressed as
\begin{eqnarray*}
d_{\mathbf{Z}_N}(a_{n-1},0)>d_{\mathbf{Z}_N}(0,a_0).
\end{eqnarray*}
Since the two distances $d_{\mathbf{Z}_N}$ and $d_{S^1}$ are in proportion through the map $\chi_N$ by (1.2), this inequality is equivalent to
\begin{eqnarray*}
d_{S^1}(\chi_N(a_{n-1}),1)>d_{S^1}(1,\chi_N(a_0)).
\end{eqnarray*}
This condition holds if and only if the midpoint $av_{S^1}(\chi_N(a_{n-1}),\chi_N(a_0))$ lies on the $S^1$-interval $[\chi_N(a_{n-1}), 1)_{S^1}$. Moreover the last condition is equivalent to $av_{\mu_N}(\chi_N(a_{n-1}),\chi_N(a_0))\in [\chi_N(a_{n-1}), 1)_{S^1}$. It follows from Proposition 1.1 that this is equivalent to the condition that $av_{\mathbf{Z}_N}(a_{n-1},a_0)\in [a_{n-1}, 0)_{\mathbf{Z}_N}$, which together with (2.20) says that $Rav(\mathbf{a})\in\mathbf{I}_N^n$. \\
(2) If $\mathbf{a}=(a_0,\cdots,a_{n-1})\in\mathbf{I}_N^n$ is improper, then it follows from the proof for the point (1) that
\begin{eqnarray}
0\leq b_{n-1}< a_0\leq b_0< a_1\leq b_1< a_2\leq \cdots \leq b_{n-2}< a_{n-1}\leq N-1.
\end{eqnarray}
Therefore $rot(\mathbf{b})=(b_{n-1}, b_0, b_1,  \cdots,b_{n-2})\in \mathbf{I}_N^n$. This completes the proof.
\end{proof}

\noindent
The following lemma is a rephrasing of Proposition 2.7.

\begin{cor}
For any increasing rhythm $\mathbf{a}\in\mathbf{I}_N^n$, we introduce the number $p(\mathbf{a})\in\{0,1\}$ by the rule
\begin{eqnarray*}
p(\mathbf{a})=
\left\{
\begin{array}{ll}
0, & \mbox{ if }\mathbf{a}\mbox{ is proper,}\\
1, & \mbox{ if }\mathbf{a}\mbox{ is improper.}\\
\end{array}
\right.
\end{eqnarray*}
Then we have
\begin{eqnarray*}
rot^{p(\mathbf{a})}(Rav(\mathbf{a}))\in\mathbf{I}_N^n.
\end{eqnarray*}
\end{cor}

Now we can define $\mathbf{I}$-version of the discrete average transformation as follows:

\begin{df}
For any $\mathbf{a}\in\mathbf{I}_N^n$, let
\begin{eqnarray*}
Iav(\mathbf{a})=rot^{p(\mathbf{a})}(Rav(\mathbf{a})).
\end{eqnarray*}
\end{df}

\noindent
It follows from Corollary 2.2 that $Iav$ defines a self map on $\mathbf{I}_N^n$:
\begin{eqnarray*}
Iav:\mathbf{I}_N^n\rightarrow\mathbf{I}_N^n.
\end{eqnarray*}
The following proposition shows that $Iav$ is compatible with $\overline{R}av$:

\begin{prp}
We have the following equality
\begin{eqnarray}
\pi|_{\mathbf{I}_N^n}\circ Iav=\overline{R}av\circ\pi|_{\mathbf{I}_N^n},
\end{eqnarray}
namely, the diagram
\begin{eqnarray*}
  \begin{CD}
     \mathbf{I}_N^n @>{\pi|_{\mathbf{I}_N^n}}>> \overline{\mathbf{R}}_N^n \\
  @V{Iav}VV @V{\overline{R}av}VV \\
     \mathbf{I}_N^n @>{\pi|_{\mathbf{I}_N^n}}>> \overline{\mathbf{R}}_N^n \\
  \end{CD}
\end{eqnarray*}
commutes.
\end{prp}

\begin{proof}
Let $\mathbf{a}$ be an arbitrary element of $\mathbf{I}_N^n$. When $\mathbf{a}$ is proper, it follows from Definition 2.10 and Corollary 2.2 that $Iav(\mathbf{a})=Rav(\mathbf{a})$. Hence the equality (2.22) follows from (2.4). When $\mathbf{a}$ is improper, it follows from Definition 2.10 that $Iav(\mathbf{a})=rot(Rav(\mathbf{a}))$. Hence we have
\begin{eqnarray*}
(\pi|_{\mathbf{I}_N^n}\circ Iav)(\mathbf{a})&=&\pi|_{\mathbf{I}_N^n}(rot(Rav(\mathbf{a})))\\
&=&\pi(rot(Rav(\mathbf{a})))\\
&=&\pi(Rav(\mathbf{a}))\hspace{5mm}(\mbox{by the definition of the quotient map }\pi)\\
&=&(\overline{R}av\circ \pi)(\mathbf{a})\hspace{3mm}(\mbox{by }(2.4))\\
&=&(\overline{R}av\circ \pi|_{\mathbf{I}_N^n})(\mathbf{a}).
\end{eqnarray*}
This completes the proof.
\end{proof}

\noindent
Since $\pi|_{\mathbf{I}_N^n}:\mathbf{I}_N^n\rightarrow \overline{\mathbf{R}}_N^n$ and $s:\overline{\mathbf{R}}_N^n\rightarrow \mathbf{I}_N^n$ are inverse to each other, this proposition implies the following:

\begin{cor}
We have the following equality
\begin{eqnarray}
Iav\circ s=s\circ\overline{R}av,
\end{eqnarray}
namely, the diagram
\begin{eqnarray*}
  \begin{CD}
     \mathbf{I}_N^n @<<{s}< \overline{\mathbf{R}}_N^n \\
  @V{Iav}VV @V{\overline{R}av}VV \\
     \mathbf{I}_N^n @<<{s}< \overline{\mathbf{R}}_N^n \\
  \end{CD}
\end{eqnarray*}
commutes.
\end{cor}

\noindent
This corollary in turn implies the following:

\begin{cor}
We have the following equality
\begin{eqnarray}
Iav\circ pr_{\mathbf{I}}=pr_{\mathbf{I}}\circ Rav,
\end{eqnarray}
namely, the diagram
\begin{eqnarray*}
  \begin{CD}
     \mathbf{I}_N^n @<<{pr_{\mathbf{I}}}< \mathbf{R}_N^n \\
  @V{Iav}VV @V{Rav}VV \\
     \mathbf{I}_N^n @<<{pr_{\mathbf{I}}}< \mathbf{R}_N^n \\
  \end{CD}
\end{eqnarray*}
commutes.
\end{cor}

\begin{proof}
Since $pr_{\mathbf{I}}=s\circ \pi$, we can compute as follows:
\begin{eqnarray*}
Iav\circ pr_{\mathbf{I}}&=&Iav\circ s\circ\pi\\
&=&s\circ\overline{R}av\circ\pi\hspace{5mm}(\mbox{by }(2.23))\\
&=&s\circ\pi \circ Rav\hspace{5mm}(\mbox{by }(2.4))\\
&=&pr_{\mathbf{I}} \circ Rav\hspace{5mm}(\mbox{by }(2.4)).
\end{eqnarray*}
This completes the proof.
\end{proof}

\subsection{Counterpart of $Rav$ on $\mathbf{B}_N^n$}
We transport the discrete average on $\overline{\mathbf{R}}_N^n$ to the corresponding map on $\mathbf{B}_N^n$:

\begin{df}
The {\rm Boolean average} $Bav:\mathbf{B}_N^n\rightarrow\mathbf{B}_N^n$ is defined by
\begin{eqnarray}
Bav=\overline{R}toB \circ \overline{R}av\circ Bto\overline{R},
\end{eqnarray}
which makes the following diagram commute:
\begin{eqnarray*}
  \begin{CD}
     \overline{\mathbf{R}}_N^n @>{\overline{R}toB}>> \mathbf{B}_N^n \\
  @V{\overline{R}av}VV     @V{Bav}VV \\
     \overline{\mathbf{R}}_N^n @>{\overline{R}toB}>> \mathbf{B}_N^n \\
  \end{CD}
\end{eqnarray*}
Namely the equality
\begin{eqnarray}
Bav\circ\overline{R}toB = \overline{R}toB\circ\overline{R}av
\end{eqnarray}
holds.
\end{df}

\noindent
The following lemma relates $Bav$ with $Iav$:

\begin{lem}
\begin{eqnarray}
Bav=ItoB \circ Iav\circ BtoI.
\end{eqnarray}
\end{lem}

\begin{proof}
This is proved by inserting the defining equations (2.15), (2.16) into (2.25) as follows:
\begin{eqnarray*}
Bav&=&\overline{R}toB \circ \overline{R}av\circ Bto\overline{R} \\
&=&ItoB \circ s \circ \overline{R}av\circ \pi \circ BtoI \hspace{10mm}(\mbox{by }(2.15), (2.16))\\
&=&ItoB \circ s \circ \pi \circ Iav\circ BtoI \hspace{10mm}(\mbox{by }(2.22))\\
&=&ItoB \circ Iav\circ BtoI.\hspace{20mm}(\mbox{by Proposition }2.4)
\end{eqnarray*}
This finishes the proof.
\end{proof}

\noindent
Since $ItoB$ is the inverse of $BtoI$, we have the following:

\begin{prp}
We have the equality
\begin{eqnarray}
Bav\circ ItoB = ItoB\circ Iav,
\end{eqnarray}
namely, the diagram
\begin{eqnarray*}
  \begin{CD}
     \mathbf{I}_N^n @>{ItoB}>> \mathbf{B}_N^n \\
  @V{Iav}VV     @V{Bav}VV \\
     \mathbf{I}_N^n @>{ItoB}>> \mathbf{B}_N^n \\
  \end{CD}
\end{eqnarray*}
commutes.
\end{prp}

\noindent
Thus we have accomplised the purpose of this section, namely our construction of three maps $\overline{R}av, Iav, Bav$, each of which corresponds to the discrete average map $Rav$ on $\mathbf{R}_N$.\\
The following example indicates how we can compute $Iav$ and $Bav$:\\
Example 2.1. Let $N=8$ and $n=3$. For the vector $\mathbf{v}=(0,0,1,1,0,0,0,1)\in\mathbf{B}_8^3$, we can compute $Bav(\mathbf{v})$ as follows:
\begin{eqnarray*}
Bav(\mathbf{v})&=&(ItoB\circ Iav \circ BtoI )(0,0,1,1,0,0,0,1)\hspace{5mm}(\mbox{by Lemma 2.1)}\\
&=&(ItoB(Iav(2,3,7))\hspace{35mm}(\mbox{by (2.12)})\\
&=&ItoB(rot(Rav(2,3,7)))\hspace{15mm}(\mbox{since }(2,3,7)\mbox{ is improper})\\
&=&ItoB(rot(2,5,0))\\
&=&ItoB(0,2,5)\\
&=&(1,0,1,0,0,1,0,0),
\end{eqnarray*}
the last quality coming from (2.11) and Definition 2.7.

\section{Cyclicity of the components of the Boolean average}
Our main theme in this paper is to investigate the properties of the Boolean average $Bav$. As a map from $\mathbf{F}_2^N$ to itself, $Bav$ has $N$ components each of which is a Boolean function on $\mathbf{F}_2^N$. In this section we show that only one of these $N$ Boolean functions determines the shapes of the other $N-1$ functions. For this purpose we need to check the functorialities of several maps.

\subsection{Boolean counterpart of the translation map $tr$ on $\mathbf{R}_N^n$}

In [2], we introduce a self map $tr:\mathbf{R}_N^n\rightarrow\mathbf{R}_N^n$ when $n\geq 2$, which is defined by
\begin{eqnarray*}
tr(a_0,\cdots,a_{n-1})=(a_0+_N1,\cdots,a_{n-1}+_N1)
\end{eqnarray*}
for any $(a_0,\cdots,a_{n-1})\in\mathbf{R}_N^n$. When $n=1$ or $n=0$, we supply this with the following natural ones:
\begin{eqnarray*}
tr(a_0)&=&(a_0+_N1),\\
tr(\phi)&=&\phi.
\end{eqnarray*}
Since this map $tr$ commutes with the map $rot$ for any $n\geq 0$, It descends to a self map on $\overline{\mathbf{R}}_N^n$, which we denote by $\overline{R}tr$:
\begin{eqnarray*}
&&\overline{R}tr: \overline{\mathbf{R}}_N^n\rightarrow \overline{\mathbf{R}}_N^n.
\end{eqnarray*}
\noindent
Therefore we have
\begin{eqnarray}
\overline{R}tr\circ \pi&=&\pi\circ tr,
\end{eqnarray}
namely, we have the following commutative diagram:
\begin{eqnarray*}
  \begin{CD}
     \mathbf{R}_N^n @>{\pi}>> \overline{\mathbf{R}}_N^n \\
  @V{tr}VV @V{\overline{R}tr}VV \\
     \mathbf{R}_N^n @>{\pi}>> \overline{\mathbf{R}}_N^n \\
   \end{CD}
\end{eqnarray*}

\noindent
Next we show that the map $tr$ commutes with $Rav$:

\begin{prp}
The two self maps $tr$ and $Rav$ on $\mathbf{R}_N^n$ commute with each other:
\begin{eqnarray*}
tr\circ Rav=Rav\circ tr,
\end{eqnarray*}
namely, we have the following commutative diagram:
\begin{eqnarray*}
  \begin{CD}
     \mathbf{R}_N^n @>{tr}>> \mathbf{R}_N^n \\
  @V{Rav}VV @V{Rav}VV \\
     \mathbf{R}_N^n @>{tr}>> \mathbf{R}_N^n \\
  \end{CD}
\end{eqnarray*}
\end{prp}

\begin{proof}
It follows from the definition of the $\mathbf{Z}_N$-average that
\begin{eqnarray}
av_{\mathbf{Z}_N}(a+_N1,b+_N1)&=&av_{\mathbf{Z}_N}(a,b)+_N1.
\end{eqnarray}
Therefore for any $\mathbf{a}=(a_0,\cdots, a_{n-1})\in\mathbf{R}_N^n$, we have
\begin{eqnarray*}
(Rav\circ tr)(\mathbf{a})&=&Rav(tr(a_0,\cdots,a_{n-1}))\\
&=&Rav(a_0+_N1,\cdots,a_{n-1}+_N1)\\
&=&(av_{\mathbf{Z}_N}(a_0+_N1,a_1+_N1),\cdots,av_{\mathbf{Z}_N}(a_{n-1}+_N1,a_0+_N1))\\
&=&(av_{\mathbf{Z}_N}(a_0,a_1)+_N1,\cdots,av_{\mathbf{Z}_N}(a_{n-1},a_0)+_N1)\hspace{5mm}(\mbox{by } (3.2))\\
&=&(tr\circ Rav)(\mathbf{a}).
\end{eqnarray*}
This completes the proof.
\end{proof}

\noindent
Since we have seen that both maps $tr$ and $Rav$ on $\mathbf{R}_N^n$ descend to the corresponding maps on $\overline{\mathbf{R}}_N^n$, this proposition implies the following:

\begin{cor}
For the two self maps $\overline{R}av$ and $\overline{R}tr$, the equality
\begin{eqnarray}
\overline{R}av\circ\overline{R}tr=\overline{R}tr\circ\overline{R}av,
\end{eqnarray}
holds, namely, the following diagram commutes:
\begin{eqnarray*}
  \begin{CD}
     \overline{\mathbf{R}}_N^n @>{\overline{R}tr}>> \overline{\mathbf{R}}_N^n \\
  @V{\overline{R}av}VV @V{\overline{R}av}VV \\
     \overline{\mathbf{R}}_N^n @>{\overline{R}tr}>> \overline{\mathbf{R}}_N^n \\
  \end{CD}
\end{eqnarray*}
\end{cor}

Next we define a Boolean counterpart $Btr$ of $tr$:

\begin{df}
Let $Btr:\mathbf{B}_N\rightarrow\mathbf{B}_N$ denote the map defined by
\begin{eqnarray*}
Btr(v_0,v_1,v_2,\cdots,v_{N-1})=(v_{N-1},v_0,v_1,\cdots,v_{N-2})
\end{eqnarray*}
for any $\mathbf{v}=(v_0,v_1,v_2,\cdots,v_{N-1})\in\mathbf{B}_N$.
\end{df}

\noindent
The following proposition shows the compatibility of $Btr$ and $tr$:

\begin{prp}
We have the equality of maps:
\begin{eqnarray}
Btr\circ RtoB=RtoB\circ tr,
\end{eqnarray}
namely, the following diagram commutes:
\begin{eqnarray*}
  \begin{CD}
     \mathbf{R}_N^n @>{RtoB}>> \mathbf{B}_N^n \\
  @V{tr}VV @V{Btr}VV \\
     \mathbf{R}_N^n @>{RtoB}>> \mathbf{B}_N^n \\
  \end{CD}
\end{eqnarray*}

\end{prp}

\begin{proof}
For any $\mathbf{a}=(a_0, a_1,\cdots, a_{n-1})\in\mathbf{R}_N^n$, Definition 2.7 implies that
\begin{eqnarray*}
RtoB(\mathbf{a})=\mathbf{v}=(v_0,v_1,\cdots,v_{N-1})\in \mathbf{B}_N^n,
\end{eqnarray*}
where
\begin{eqnarray*}
v_i=\left\{
\begin{array}{ll}
1, & \mbox{ if $i\in \mathbf{a}$},\\
0, & \mbox{ if $i\not\in \mathbf{a}$}, \\
\end{array}
\right.
\end{eqnarray*}
Therefore the left hand side of (3.4) maps $\mathbf{a}$ to
\begin{eqnarray*}
Btr(RtoB(\mathbf{a}))&=&Btr(v_0,v_1,\cdots,v_{N-1})\\
&=&(v_{N-1},v_0,\cdots,v_{N-2}).
\end{eqnarray*}
Let $w_i=v_{i-_N1}$ for any $i\in\mathbf{Z}_N$. Then we have
\begin{eqnarray*}
\mathbf{w}=(w_0,w_1,\cdots,w_{N-1})=(v_{N-1},v_0,\cdots,v_{N-2}),
\end{eqnarray*}
namely we have 
\begin{eqnarray*}
\mathbf{w}=Btr(RtoB(\mathbf{a})).
\end{eqnarray*}
Notice that we have the following series of equivalences:
\begin{eqnarray*}
w_i=1 &\Leftrightarrow& v_{i-_N1}=1 \\
&\Leftrightarrow& i-_N1\in \mathbf{a} \\
&\Leftrightarrow& i\in tr(\mathbf{a}). 
\end{eqnarray*}
It follows that
\begin{eqnarray*}
\mathbf{w}=RtoB(tr(\mathbf{a})).
\end{eqnarray*}
Hence we have
\begin{eqnarray*}
Btr\circ RtoB=RtoB\circ tr,
\end{eqnarray*}
This completes the proof.
\end{proof}

\noindent
As an immediate consequence, we have the following:

\begin{cor}
The equality of composite maps
\begin{eqnarray}
Btr\circ \overline{R}toB=\overline{R}toB\circ \overline{R}tr
\end{eqnarray}
holds, namely, the following diagram commutes:
\begin{eqnarray*}
  \begin{CD}
     \overline{\mathbf{R}}_N^n @>{\overline{R}toB}>> \mathbf{B}_N^n \\
  @V{\overline{R}tr}VV @V{Btr}VV \\
     \overline{\mathbf{R}}_N^n @>{\overline{R}toB}>> \mathbf{B}_N^n \\
  \end{CD}
\end{eqnarray*}
\end{cor}

\begin{proof}
For any $x\in\overline{\mathbf{R}}_N^n$, let $\mathbf{a}$ be an arbitrary representative of the equivalence class $x$. Then we have
\begin{eqnarray*}
(\overline{R}toB\circ \overline{R}tr)(x)&=&(\overline{R}toB\circ \overline{R}tr)(\pi(\mathbf{a}))\\
&=&(\overline{R}toB\circ \overline{R}tr \circ \pi)(\mathbf{a})\\
&=&(\overline{R}toB\circ \pi \circ tr)(\mathbf{a})\hspace{10mm}(\mbox{by }(3.1))\\
&=&(RtoB\circ tr)(\mathbf{a})\hspace{16mm}(\mbox{by Proposition }2.6)\\
&=&(Btr\circ RtoB)(\mathbf{a})\hspace{13mm}(\mbox{by }(3.4))\\
&=&(Btr\circ \overline{R}toB\circ \pi)(\mathbf{a})\hspace{8mm}(\mbox{by Propotision }2.6)\\
&=&(Btr\circ \overline{R}toB)(x).
\end{eqnarray*}
This completes the proof of Corollary 3.2.
\end{proof}

\subsection{Cyclicity of the Boolean average}
Now we can prove the main result of this section, which asserts that only one of the components of $Bav$ determines the others completely:

\begin{thm}
$(1)$ The two self maps $Bav, Btr$ of the space $\mathbf{B}_N^n$ commute with each other:
\begin{eqnarray}
Btr\circ Bav=Bav\circ Btr,
\end{eqnarray}
namely, the following diagram commutes:
\begin{eqnarray*}
  \begin{CD}
     \mathbf{B}_N^n @>{Btr}>> \mathbf{B}_N^n \\
  @V{Bav}VV @V{Bav}VV \\
     \mathbf{B}_N^n @>{Btr}>> \mathbf{B}_N^n \\
  \end{CD}
\end{eqnarray*}
$(2)$ For any $i\in\mathbf{Z}_N$, we denote by $Bav_N^i$ the composite $pr_i\circ Bav:\mathbf{B}_N^n\rightarrow\mathbf{F}_2$, where $pr_i:\mathbf{F}_2^N\rightarrow\mathbf{F}_2$ is the projection onto the $i$-th factor. Then we have
\begin{eqnarray}
Btr^*(Bav_N^i)=Bav_N^{i-_N1}.
\end{eqnarray}
Here the left hand side means the pull-back of the function $Bav_N^i$ through the self map $Btr$ on $\mathbf{B}_N^n$, namely the composite $Bav_N^i\circ Btr$.
\end{thm}

\begin{proof}
(1) Composing the left hand side of (3.6) with the bijection $\overline{R}toB$, we compute as follows:
\begin{eqnarray*}
Btr\circ Bav\circ\overline{R}toB&=&Btr\circ\overline{R}toB\circ \overline{R}av\hspace{10mm}(\mbox{by }(2.26))\\
&=&\overline{R}toB\circ \overline{R}tr\circ \overline{R}av\hspace{10mm}(\mbox{by }(3.5))\\
&=&\overline{R}toB\circ \overline{R}av\circ \overline{R}tr\hspace{10mm}(\mbox{by }(3.3))\\
&=&Bav\circ \overline{R}toB\circ \overline{R}tr\hspace{10mm}(\mbox{by }(2.26))\\
&=&Bav\circ Btr\circ \overline{R}toB\hspace{10mm}(\mbox{by }(3.5))\\
\end{eqnarray*}
This shows that the equality (3.6) holds.\\
(2) Note that we have
\begin{eqnarray}
pr_i\circ Btr=pr_{i-_N1},
\end{eqnarray}
since for any $\mathbf{v}=(v_0,v_1,\cdots,v_{N-1})\in\mathbf{B}_N^n$, we have
\begin{eqnarray*}
(pr_i\circ Btr)(\mathbf{v})&=&pr_i(v_{N-1},v_0,v_1,\cdots,v_{N-2})\\
&=&v_{i-_N1}\\
&=&pr_{i-_N1}(\mathbf{v}).
\end{eqnarray*}
Hence we can compute as follows:
\begin{eqnarray*}
Btr^*(Bav_N^i)&=&Bav_N^i\circ Btr\\
&=&pr_i\circ Bav\circ Btr\\
&=&pr_i\circ Btr\circ Bav \hspace{12mm}(\Leftarrow \mbox{ by }(3.6))\\
&=&pr_{i-_N1}\circ Bav \hspace{15mm}(\Leftarrow \mbox{ by }(3.8))\\
&=&Bav_N^{i-_N1}.
\end{eqnarray*}
This completes the proof.
\end{proof}

\begin{rem}
In concrete terms, the equality (3.7) means that every $Bav_N^i$ $(i\in\mathbf{Z}_N)$ can be obtained by a cyclic coordinate change from $Bav_N^0$.
\end{rem}

\section{Polynomial expression of the Boolean average}
In this section we investigate how the Boolean average $Bav$ is expressed as a polynomial map on $\mathbf{B}_N$. First we recall some standard method which translates any propositional functions into polynomial expressions.

\subsection{Algebraic standard form of propositional functions}
In this subsection we recall some standard facts about propositional functions. See [3,4] for details.\\
Let the truth-values $0$ and $1$ stand for "false" and "true", respectively. Then any proposition can be regarded as an $\mathbf{F}_2$-valued function on $\mathbf{F}_2^N$ for an appropriate $N$. For example, the logical connective "and" corresponds to the multiplication operator "$\times$", as is observed from the truth table of "and". On the other hand, the addition operator "$+$" corresponds to the connective "xor (exclusive or)". Since any proposition can be expressed as a compound of these two connectives, it corresponds to a Boolean function of a certain number of variables. In view of the fact that we express a polynomial in several variables usually as a sum of monomials, the disjunctive normal form of logical formula plays a fundamental role for our purpose. We recall its construction by a typical example.\\
Example 4.1. The disjunctive normal form of $P=(x \vee y)\wedge (y\vee z)$. The truth table of $P$ is given by
\begin{eqnarray*}
\begin{array}{lccccccccc}
\mbox{number} & \vline & x & y & z & \vline & x\vee y & y\vee z & \vline &  P \\
\hline
(1) & \vline & 0 & 0 & 0 & \vline & 0 & 0 & \vline &  0 \\
(2) & \vline & 0 & 0 & 1 & \vline & 0 & 1 & \vline &  0 \\
(3) & \vline & 0 & 1 & 0 & \vline & 1 & 1 & \vline &  1 \\
(4) & \vline & 0 & 1 & 1 & \vline & 1 & 1 & \vline &  1 \\
(5) & \vline & 1 & 0 & 0 & \vline & 1 & 0 & \vline &  0 \\
(6) & \vline & 1 & 0 & 1 & \vline & 1 & 1 & \vline &  1 \\
(7) & \vline & 1 & 1 & 0 & \vline & 1 & 1 & \vline &  1 \\
(8) & \vline & 1 & 1 & 1 & \vline & 1 & 1 & \vline &  1 \\
\end{array}
\end{eqnarray*}
\begin{center}
Table 4.1. Truth table for $P=(x \vee y)\wedge (y\vee z)$
\end{center}
There are five rows, named (3), (4), (6), (7), (8), for which the value of $P$ equals 1. We associate a term with each of these five rows. For example, the third row gives $P$ the value 1. We make a constituent of the corresponding term by the following rule. Here "$w$" stands for any one of the variables $x,y,z$:
\begin{eqnarray*}
&&\mbox{{\it If the value of $w$ is 1, then use "$w$", }}\\
&&\mbox{{\it otherwise use "$\overline{w}$" as its constituent.}}
\end{eqnarray*}
Therefore the third row provides us with the term $\overline{x}y\overline{z}$, since $(x,y,z)=(0,1,0)$. Connecting the term $\overline{x}y\overline{z}$ with the other four terms obtained from the remaining four rows, we obtain the following disjunctive normal form of the proposition $P$:
\begin{eqnarray*}
P=\overline{x}y\overline{z}\vee \overline{x}yz \vee x\overline{y}z \vee xy\overline{z} \vee xyz.
\end{eqnarray*}
Note that we can replace all "$\vee$" by "$\underline{\vee}$", since each term on the right hand side contradicts to each other. Therefore we can rewrite this as
\begin{eqnarray*}
P=\overline{x}y\overline{z}\hspace{1mm}\underline{\vee}\hspace{1mm} \overline{x}yz \hspace{1mm}\underline{\vee}\hspace{1mm} x\overline{y}z \hspace{1mm}\underline{\vee}\hspace{1mm} xy\overline{z} \hspace{1mm}\underline{\vee}\hspace{1mm} xyz.
\end{eqnarray*}
By letting $\overline{w}=1+w$ for each variable $w\in\{x,y,z\}$, and by reducing the result modulo 2, we arrive at the following expression $B_P$ of $P$ as a Boolean polynomial:
\begin{eqnarray*}
B_P&=&(1+x)y(1+z)+(1+x)yz+x(1+y)z+xy(1+z)+xyz\\
&=&5xyz+2xy+2yz+xz+y\\
&=&xyz+xz+y.
\end{eqnarray*}
One can check that the rightmost side $xyz+xz+y$ coincides with $P$ as a function on $\mathbf{F}_2^3$.

\begin{rem}
In general, any Boolean polynomial $f(x_1,\cdots,x_n)$ should be regarded as an object in $\mathbf{F}_2[x_1,\cdots,x_n]/I$, where $I$ denotes the ideal
\begin{eqnarray*}
I=(x_1^2-x_1,\cdots,x_n^2-x_n).
\end{eqnarray*}
In the example above, however, one does not need to reduce $B_P$ modulo the ideal $I$, since it has no squared variables.
\end{rem}

\subsection{Rhythm polynomials}
In this subsection we investigate the problem how the self map $Bav$ on $\mathbf{B}_N^n$ can be expressed as a polynomial map. Since $\mathbf{B}_N=\bigcup_{n=0}^N \mathbf{B}_N^n$, we can extend the domain of the definition of the map $Bav$ to the whole $\mathbf{B}_N$ naturally. By abuse of language we also denote the extended map by $Bav$. Furthermore the composite 
\begin{eqnarray*}
Bav_N^i=pr_i\circ Bav:\mathbf{F}_2^N\rightarrow\mathbf{F}_2
\end{eqnarray*}
is called the $i$-th {\it rhythm polynomial modulo $N$}. By Theorem 3.1, (2), every rhythm polynomial is obtained from the 0-th rhythm polynomial $Bav_N^0$ through a cyclic change of variables $v_i$, $i\in\mathbf{Z}_N$. Therefore we can focus our attention exclusively on $Bav_N^0$.\\
First we examine the problem for the case $N=3$. The following table displays the Boolean averages of vectors in $\mathbf{B}_3$. Recall that $Bav=ItoB\circ Iav\circ BtoI$ by (2.27):
\begin{eqnarray*}
\begin{array}{ccccccc}
\mathbf{v} & \vline & BtoI(\mathbf{v}) & \vline & Iav(BtoI(\mathbf{v})) & \vline & Bav(\mathbf{v}) \\
\hline
(0,0,0) & \vline & \phi & \vline & \phi & \vline & (0,0,0) \\
(0,0,1) & \vline & (2) & \vline & (2) & \vline & (0,0,1) \\
(0,1,0) & \vline & (1) & \vline & (1) & \vline & (0,1,0) \\
(0,1,1) & \vline & (1,2) & \vline & (0,1) & \vline & (1,1,0) \\
(1,0,0) & \vline & (0) & \vline & (0) & \vline & (1,0,0) \\
(1,0,1) & \vline & (0,2) & \vline & (1,2) & \vline & (0,1,1) \\
(1,1,0) & \vline & (0,1) & \vline & (0,2) & \vline & (1,0,1) \\
(1,1,1) & \vline & (0,1,2) & \vline & (0,1,2) & \vline & (1,1,1) \\
\end{array}
\end{eqnarray*}
\begin{center}
Table 4.2. Boolean averages of three-dimensional vectors
\end{center}
\noindent
Employing the algorithm explained in the previous subsection, we have
\begin{eqnarray*}
Bav_3^0(v_0,v_1,v_2)&=&\overline{v_0}v_1v_2+v_0\overline{v_1}\overline{v_2}+v_0v_1\overline{v_2}+v_0v_1v_2\\
&=&v_0+v_0v_2+v_1v_2,\\
Bav_3^1(v_0,v_1,v_2)&=&\overline{v_0}v_1\overline{v_2}+\overline{v_0}v_1v_2+v_0\overline{v_1}v_2+v_0v_1v_2\\
&=&v_1+v_0v_1+v_0v_2,\\
Bav_3^2(v_0,v_1,v_2)&=&\overline{v_0}\overline{v_1}v_2+v_0\overline{v_1}v_2+v_0v_1\overline{v_2}+v_0v_1v_2\\
&=&v_2+v_0v_1+v_1v_2.
\end{eqnarray*}
This computation confirms the cyclicity stated in Theorem 3.1, (2).\\
In a similar way we can compute the rhythm polynomial $Bav_N^0$ for small $N$. We list the results as follows:
\begin{eqnarray*}
\begin{array}{cl}
N & Bav_N^0 \\
\hline
3 & v_0+v_0v_2+v_1v_2 \\
4 & v_0+v_0v_2+v_0v_3+v_1v_3+v_2v_3+v_0v_1v_2+v_1v_2v_3 \\
5 & v_0+v_0v_2+v_0v_3+v_0v_4+v_1v_4+v_2v_3+v_2v_4+v_0v_1v_2+v_0v_1v_3\\
& +v_0v_3v_4+v_1v_2v_3+v_1v_2v_4+v_2v_3v_4+v_0v_1v_3v_4+v_1v_2v_3v_4 \\
\end{array}
\end{eqnarray*}
\begin{center}
Table 4.3. $Bav_N^0$ for $N=3,4,5$ in variables $v_i$
\end{center}
\noindent
We observe here that for each $N\in\{3,4,5\}$ the number of terms in $Bav_N^0$ is equal to $2^{N-1}-1$. If we change, however, the variables $v_i$ to $w_i+1$ for $i\in[0,N-1]$, the number of terms diminishes to $2(N-1)$, as is observed in the following list:
\begin{eqnarray*}
\begin{array}{cl}
N & Bav_N^0 \\
\hline
3 & 1+w_1+w_0w_2+w_1w_2 \\
4 & 1+w_1+w_0w_1+w_0w_3+w_0w_1w_2+w_1w_2w_3 \\
5 & 1+w_1+w_0w_1+w_0w_4+w_0w_1w_2+w_0w_1w_4\\
&+w_0w_1w_3w_4+w_1w_2w_3w_4 \\
6 & 1+w_1+w_0w_1+w_0w_5+w_0w_1w_2+w_0w_1w_5\\
&+w_0w_1w_2w_5+w_0w_1w_4w_5+w_0w_1w_2w_3w_5+w_1w_2w_3w_4w_5 \\
\end{array}
\end{eqnarray*}
\begin{center}
Table 4.4. $Bav_N^0$ for $N=3,4,5, 6$ in variables $w_i$
\end{center}

\noindent
At this point, we can imagine neither the general shape of $Bav_N^0$, nor a possible relation between $Bav_N^0$ and $Bav_{N+1}^0$. We will show that, if we change the index set $\mathbf{Z}_N$ to $\mathbf{Z}_{N,\pm}$, the set of the least absolute remainders modulo $N$, as a new index set, then we can detect the structure of $Bav_N^0$ more easily. When $N$ is even, we choose $\frac{N}{2}$ from the two possible representatives $\{-\frac{N}{2},\frac{N}{2}\}$ as the remainder, and hence the set $\mathbf{Z}_{N,\pm}$ is specified as
\begin{eqnarray*}
\mathbf{Z}_{N,\pm}=
\left\{
\begin{array}{ll}
\{-m,-m+1,\cdots,-1,0,1,\cdots, m-1, m\}, & \mbox{ if $N=2m+1$}, \\
\{-m+1,\cdots,-1,0,1,\cdots, m-1, m\}, & \mbox{ if $N=2m$}. \\
\end{array}
\right.
\end{eqnarray*}
By employing the floor symbol, we can express the set $\mathbf{Z}_{N,\pm}$ regardless of the parity of $N$ as follows:
\begin{eqnarray}
\mathbf{Z}_{N,\pm}=[-\left\lfloor\frac{N-1}{2}\right\rfloor,\left\lfloor\frac{N}{2}\right\rfloor].
\end{eqnarray}
For later use, we divide the set $\mathbf{Z}_{N,\pm}$ into the following three parts:
\begin{eqnarray*}
\mathbf{Z}_{N,\pm}=\mathbf{Z}_{N,-}\sqcup\{0\}\sqcup\mathbf{Z}_{N,+},
\end{eqnarray*}
where
\begin{eqnarray*}
\mathbf{Z}_{N,-}&=&\{k\in\mathbf{Z}_{N,\pm};k<0\},\\
\mathbf{Z}_{N,+}&=&\{k\in\mathbf{Z}_{N,\pm};k>0\},
\end{eqnarray*}
and put
\begin{eqnarray*}
\mathbf{Z}_{N,\leq 0}&=&\mathbf{Z}_{N,-}\cup\{0\}.
\end{eqnarray*}

We construct a bridge between $\mathbf{Z}_N$ and $\mathbf{Z}_{N,\pm}$ employing a simple bijection defined as follows:

\begin{df}
Let $rem_{\pm}:\mathbf{Z}\rightarrow \mathbf{Z}_{N,\pm}$ denote the map whose value at $k\in\mathbf{Z}$ is the least absolute remainder of $k$ modulo $N$. Let $\varphi_N:\mathbf{Z}_N\rightarrow\mathbf{Z}_{N,\pm}$ denote the composite $rem_{\pm}\circ \iota$, where $\iota:\mathbf{Z}_N\rightarrow \mathbf{Z}$ is the natural inclusion map, namely, $\varphi_N$ is a bijection between $\mathbf{Z}_N$ and $\mathbf{Z}_{N,\pm}$ defined by
\begin{eqnarray*}
\varphi_N(k)=
\left\{
\begin{array}{ll}
k, & \mbox{ if }0\leq k\leq\left\lfloor\frac{N}{2}\right\rfloor, \\
k-N, & \mbox{ if }\left\lfloor\frac{N}{2}\right\rfloor<k\leq N-1. \\
\end{array}
\right.
\end{eqnarray*}
\end{df}

\begin{rem}
The map $\varphi_N$ satisfies the condition 
\begin{eqnarray}
\varphi_N(a)\equiv a\pmod N
\end{eqnarray}
for any $a\in\mathbf{Z}_N$. Therefore $\varphi_N$ exchanges the two representatives in $\mathbf{Z}_N$ and in $\mathbf{Z}_{N,\pm}$ of a class in the quotient $\mathbf{Z}/N\mathbf{Z}$.
\end{rem}

\noindent
We introduce new variables $y_j$, $j\in\mathbf{Z}_{N,\pm}$, by the rule
\begin{eqnarray*}
w_i=y_{\phi_N(i)}
\end{eqnarray*}
for any $i\in\mathbf{Z}_N$. Then Table 4.4 transforms to the following:
\begin{eqnarray*}
\begin{array}{cl}
N & Bav_N^0 \\
\hline
3 & 1+y_1+y_0y_{-1}+y_1y_{-1} \\
4 & 1+y_1+y_0y_1+y_0y_{-1}+y_0y_1y_2+y_1y_2y_{-1} \\
5 & 1+y_1+y_0y_1+y_0y_{-1}+y_0y_1y_2+y_0y_1y_{-1}\\
&+y_0y_1y_{-2}y_{-1}+y_1y_2y_{-2}y_{-1} \\
6 & 1+y_1+y_0y_1+y_0y_{-1}+y_0y_1y_2+y_0y_1y_{-1}\\
&+y_0y_1y_2y_{-1}+y_0y_1y_{-1}y_{-1} \\
& +y_0y_1y_2y_3y_{-1}+y_1y_2y_3y_{-2}y_{-1} \\
\end{array}
\end{eqnarray*}
\begin{center}
Table 4.5. $Bav_N^0$ for $N=3,4,5,6$ in variables $y_i$, $i\in\mathbf{Z}_{N,\pm}$
\end{center}

\noindent
This table suggests us the structure of the rhythm polynomial $Bav_N^0$ more clearly than Table 4.4. We explore in the next section the true nature of these polynomials.

\section{Determination of rhythm polynomials}
In this section we determine the Boolean polynomial $Bav_N^0$ completely for any $N$, and derive a recurrence formula which connects $Bav_N^0$ and $Bav_{N-1}^0$. Furthermore we prove that $Bav_N^0$ is a {\it balanced} polynomial. (See Definition 5.10.)

\subsection{Change of the set of indices}
As is motivated in Subsection 4.2, we need to use two index sets $\mathbf{Z}_N$ and $\mathbf{Z}_{N,\pm}$ properly to understand the rhythm polynomials. Fortunately enough, even if we employ $\mathbf{Z}_{N,\pm}$ as the index set, several corresponding objects can be defined in almost the same way.\\
By abuse of language, we denote also by $\varphi_N$ the bijective map $\mathbf{Z}_N^n\rightarrow\mathbf{Z}_{N,\pm}^n$ induced by $\varphi_N$ between their self-products, and put 
\begin{eqnarray*}
\mathbf{R}_{N,\pm}^n=\varphi_N(\mathbf{R}_N^n).
\end{eqnarray*}

\begin{df}
Let $\mathbf{I}_{N,\pm}^n$ denote the subset of $\mathbf{R}_{N,\pm}^n$ which consists of the monotone increasing sequence:
\begin{eqnarray*}
\mathbf{I}_{N,\pm}^n=\{\mathbf{a}=(a_0,\cdots,a_{n-1})\in\mathbf{R}_{N,\pm}^n;a_0<a_1<\cdots <a_{n-1}\}.
\end{eqnarray*}
\end{df}

\begin{df}
Let $\mathbf{B}_{N,\pm}$ denote the set of $n$-dimensional Boolean vectors with index set $\mathbf{Z}_{N,\pm}$, namely,
\begin{eqnarray*}
\mathbf{B}_{N,\pm}=\{(v_{\ell},\cdots,v_g);v_i\in\mathbf{F}_2 \mbox{ for any }i\in[\ell,g]\},
\end{eqnarray*}
where we introduce the two symbols $\ell, g$ by
\begin{eqnarray*}
\ell=-\left\lfloor\frac{N-1}{2}\right\rfloor,\hspace{2mm}g=\left\lfloor\frac{N}{2}\right\rfloor,
\end{eqnarray*}
which are the \underline{l}east and the \underline{g}reatest elements in $\mathbf{Z}_{N,\pm}$.
\end{df}

\noindent
We use the same symbol $ItoB$ and $BtoI$ as before to denote the pair of bijections between $\mathbf{I}_{N,\pm}^n$ and $\mathbf{B}_{N,\pm}^n$, which are defined as follows:

\begin{df}
$(1)$ For any $\mathbf{a}\in\mathbf{I}_{N,\pm}^n$, the Boolean vector $\mathbf{v}=(v_{\ell},\cdots,v_g)=ItoB(\mathbf{a})\in\mathbf{B}_{N,\pm}$ is defined by
\begin{eqnarray}
v_i=
\left\{
\begin{array}{ll}
1, & \mbox{ if }i\in\mathbf{a},\\
0, & \mbox{ if }i\not\in\mathbf{a}.\\
\end{array}
\right.
\end{eqnarray}
$(2)$ For any $\mathbf{v}\in\mathbf{B}_{N,\pm}^n$, the $n$-tuple $BtoI(\mathbf{v})\in\mathbf{I}_{N,\pm}^n$ is defined to be the unique arrangement of the set $\{i\in\mathbf{Z}_{N,\pm};v_i=1\}$ in increasing order.
\end{df}

\subsection{Parental pair and ancestor of zero in $\mathbf{Z}_N$}
It is important to keep in mind the following fact:

\begin{lem}
The $0$-th rhythm polynomial $Bav_N^0$ takes value $1$ at $\mathbf{v}\in\mathbf{B}_N$ if and only if 
\begin{eqnarray*}
0\in Rav(BtoI(\mathbf{v})).
\end{eqnarray*}
\end{lem}

\begin{proof}
Let $\mathbf{a}=(a_0,\cdots,a_{n-1})=BtoI(\mathbf{v})\in\mathbf{I}_N^n$. Since it is in $\mathbf{I}_N^n$, we have
\begin{eqnarray*}
Iav(\mathbf{a})=Rav(\mathbf{a}).
\end{eqnarray*}
Recall that the two maps $Iav$ and $Bav$ are compatible through the map $ItoB$ by Proposition 2.9. Hence we have the following series of equivalences:
\begin{eqnarray*}
0\in Rav(\mathbf{a})&\Leftrightarrow& 0\in Iav(\mathbf{a})\\
&\Leftrightarrow&0\in (BtoI\circ Bav\circ ItoB)(\mathbf{a})\hspace{10mm}(\mbox{by }(2.27))\\
&\Leftrightarrow& 0\in(BtoI\circ Bav)(\mathbf{v})\\
&\Leftrightarrow& 0\in supp(Bav(\mathbf{v}))\hspace{25mm}(\mbox{by }(2.12))\\
&\Leftrightarrow& pr_0(Bav(\mathbf{v}))=1\\
&\Leftrightarrow& Bav_N^0(\mathbf{v})=1.
\end{eqnarray*}
Thus we finish the proof.
\end{proof}

\noindent
By this lemma, in order to find the form of $Bav_N^0$, it is essential to determine the set of pairs $(a,b)\in\mathbf{Z}_N\times\mathbf{Z}_N$ such that $av_{\mathbf{Z}_N}(a,b)=0$. This leads us to the following:

\begin{df}
A pair $(a,b)\in\mathbf{Z}_N\times\mathbf{Z}_N$ is said to be a {\rm parental pair of zero} if $av_{\mathbf{Z}_N}(a,b)=0$. A rhythm $\mathbf{a}\in\mathbf{R}_N$ is called an ancestor of zero if $0\in Rav(\mathbf{a})$.
\end{df}

The following proposition provides us with a necessary and sufficient condition for a pair to be a parental pair of zero. It is expressed as a congruence modulo $N$. This will facilitate our transition from $\mathbf{Z}_N$ to $\mathbf{Z}_{N,\pm}$ in the next subsection.

\begin{prp}
Let $a,b$ be an arbitrary pair of elements in $\mathbf{Z}_N$. \\
$(0)$ When $a=0$, the following conditions are equivalent: \\
$(0.{\rm A}) \hspace{1mm}av_{\mathbf{Z}_N}(a,b)=0$, \\
$(0.{\rm B}) \hspace{1mm}b=0,1$. \\
$(1)$ When $a=1$, these exists no $b\in\mathbf{Z}_N$ such that $av_{\mathbf{Z}_N}(a,b)=0$.\\
$(2)$ When $a\geq 2$, the following conditions are equivalent: \\
$(2.{\rm A}) \hspace{1mm}av_{\mathbf{Z}_N}(a,b)=0$, \\
$(2.{\rm B}) \hspace{1mm} a>b\mbox{ and }a+b\equiv 0,1\pmod N$.
\end{prp}

\begin{rem}
When $n\geq 2$, an element in $\mathbf{R}_N^n$ consists of $n$ \underline{distinct} elements of $\mathbf{Z}_N$ by the definition. Hence the first alternative in $(0.{\rm B})$, namely the case when $(a,b)=(0,0)$ occurs only if $n=1$. In that case, the discrete average $Rav((0))$ of the one-element rhythm $(0)$ is $(0)$.
\end{rem}

\begin{proof}
(0) When $a=0$, we have
\begin{eqnarray*}
av_{\mathbf{Z}_N}(0,b)=\left\lfloor\frac{b}{2}\right\rfloor,
\end{eqnarray*}
and hence it is equal to $0$ if and only if $b=0,1$. Therefore (0.A) and (0.B) are equivalent. \\
(1) When $a=1$, if $av_{\mathbf{Z}_N}(1,b)=0$, then $0\in [1,b]_N$ by Proposition 1.2, (2). This occurs only if $b=0$. When $N$ is odd and $N=2m+1$, we have
\begin{eqnarray*}
av_{\mathbf{Z}_N}(1,0)&=&1+_N\left\lfloor\frac{0-_N1}{2}\right\rfloor\\
&=&1+_N\left\lfloor\frac{N-1}{2}\right\rfloor\\
&=&1+_Nm.
\end{eqnarray*}
Note that $1+_Nm=1+m<1+2m=N$, and hence $1+_Nm$ cannot be equal to 0. When $N$ is even and $N=2m$, it follows by a similar computation that $av_{\mathbf{Z}_N}(1,0)=1+_N(m-1)$, which cannot be zero. This completes the proof of the assertion (1).\\
(2) (2.A)$\Rightarrow$(2.B): Since $av_{\mathbf{Z}_N}(a,b)\in[a,b]_N$ by Proposition 1.2, (2), the condition $av_{\mathbf{Z}_N}(a,b)=0$ implies that $0\in [a,b]$, and hence $a>b$. Furthermore the equality $av_{\mathbf{Z}_N}(a,b)=a+_N\left\lfloor\frac{b-_Na}{2}\right\rfloor=0$ holds if and only if
\begin{eqnarray}
\left\lfloor\frac{b-_Na}{2}\right\rfloor=N-a,
\end{eqnarray}
since $a\neq 0$. Note that
\begin{eqnarray*}
2\left\lfloor\frac{x}{2}\right\rfloor =
\left\{
\begin{array}{ll}
x, & \mbox{ if }x\equiv 0\pmod{2}, \\
x-1, & \mbox{ if }x\equiv 1\pmod{2}, \\
\end{array}
\right.
\end{eqnarray*}
for any $x\in\mathbf{Z}$. When $b-_Na\equiv 0\pmod{2}$, doubling the both sides of (5.2), we have
\begin{eqnarray*}
b-_Na=2N-2a,
\end{eqnarray*}
which gives us the congruence
\begin{eqnarray*}
a+b\equiv 0\pmod N.
\end{eqnarray*}
When $b-_Na\equiv 1\pmod{2}$, the equality (5.2) implies similarly that
\begin{eqnarray*}
b-_Na-1=2N-2a,
\end{eqnarray*}
which gives us the congruence
\begin{eqnarray*}
a+b\equiv 1\pmod N.
\end{eqnarray*}
Thus (2.A) implies (2.B).\\
(2.B)$\Rightarrow$(2.A): In case $a+b\equiv 0\pmod N$, we have $b=N-a$ since $a>0$. Furthermore, since $a>b$, we have $a>N-a$, therefore $0>N-2a>-N$. It follows that
\begin{eqnarray*}
b-_Na&=&(N-a)-_Na\\
&=&\langle (N-a)-a\rangle_N\\
&=&\langle N-2a\rangle_N\\
&=&2N-2a.
\end{eqnarray*}
Hence we have
\begin{eqnarray*}
av_{\mathbf{Z}_N}(a,b)&=&a+_N\left\lfloor\frac{b-_Na}{2}\right\rfloor\\
&=&a+_N\left\lfloor\frac{2N-2a}{2}\right\rfloor\\
&=&a+_N(N-a)\\
&=&0.
\end{eqnarray*}
In case $a+b\equiv 1\pmod N$, we have $b=N+1-a$, since we have assumed that $a\geq 2$. Since $a>b$, we have $a>N+1-a$, therefore $0>N+1-2a>-N$. It follows that
\begin{eqnarray*}
b-_Na&=&(N+1-a)-_Na\\
&=&\langle (N+1-a)-a\rangle_N\\
&=&\langle N+1-2a\rangle_N\\
&=&2N+1-2a.
\end{eqnarray*}
Hence we have
\begin{eqnarray*}
av_{\mathbf{Z}_N}(a,b)&=&a+_N\left\lfloor\frac{b-_Na}{2}\right\rfloor\\
&=&a+_N\left\lfloor\frac{2N+1-2a}{2}\right\rfloor\\
&=&a+_N(N-a)\\
&=&0.
\end{eqnarray*}
This completes the proof.
\end{proof}

\begin{cor}
When $a\in [1,N/2]$, there exists no $b\in\mathbf{Z}_N$ such that $av_{\mathbf{Z}_N}(a,b)$\\$=0$.
\end{cor}

\begin{proof}
By Proposition 5.1, (1), we may assume that $2\leq a\leq N/2$. Suppose that $av_{\mathbf{Z}_N}(a,b)=0$ holds for some $b\in\mathbf{Z}_N$. Then it follows from Proposition 5.1, (2) that $a+b\equiv 0,1\pmod N$, which implies $b=N-a$ or $b=N+1-a$. Since $a>b$ by (2.B), we have $2a>N$ or $2a>N+1$, and hence $a>N/2$ in both cases. This contradiction finishes the proof.
\end{proof}

\subsection{Parental pair and ancestor of zero in $\mathbf{Z}_{N,\pm}$}
In this subsection, we show that, if we employ the index set $\mathbf{Z}_{N,\pm}$ instead of $\mathbf{Z}_N$, then the assertions of Proposition 5.1 and Corollary 5.1 are simplified considerably.

\begin{df}
For any $(a,b)\in\mathbf{Z}_{N,\pm}\times\mathbf{Z}_{N,\pm}$, $\mathbf{Z}_{N,\pm}$-{\rm average} $av_{\mathbf{Z}_N,\pm}(a,b)$ is defined by
\begin{eqnarray*}
av_{\mathbf{Z}_N,\pm}(a,b)=\varphi_N(av_{\mathbf{Z}_N}(\varphi_N^{-1}(a),\varphi_N^{-1}(b))).
\end{eqnarray*}
Furthermore, on $\mathbf{R}_{N,\pm}^n=\varphi_N(\mathbf{R}_N^n)$, we define {\rm $\pm$-discrete average transformation} $Rav_{\pm}:\mathbf{R}_{N,\pm}^n\rightarrow\mathbf{R}_{N,\pm}^n$ by
\begin{eqnarray*}
Rav_{\pm}(\mathbf{a})=\varphi_N(Rav(\varphi_N^{-1}(\mathbf{a})))
\end{eqnarray*}
for any $\mathbf{a}\in\mathbf{R}_{N,\pm}^n$.
\end{df}

\noindent
The two notions, parental pairs and ancestors of zero, introduced in the previous subsection, are translated as follows:

\begin{df}
A pair $(a,b)\in\mathbf{Z}_{N,\pm}\times\mathbf{Z}_{N,\pm}$ is said to be a {\rm parental pair of zero} if $rav_{\pm}(a,b)=0$. We denote the set of parental pairs of zero in $\mathbf{Z}_{N,\pm}\times\mathbf{Z}_{N,\pm}$ by $Par_N$:
\begin{eqnarray*}
Par_N=\{(a,b)\in\mathbf{Z}_{N,\pm}\times\mathbf{Z}_{N,\pm};rav_{\pm}(a,b)=0\}.
\end{eqnarray*}
A rhythm $\mathbf{a}\in\mathbf{R}_{N,\pm}$ is called an ancestor of zero if $0\in Rav_{\pm}(\mathbf{a})$.
\end{df}

We can determine the set $Par_N$ completely in the following way:

\begin{prp}
$(0)$ If $(a,b)\in Par_N$, then $a\in\mathbf{Z}_{N,\leq 0}=\mathbf{Z}_{N,-}\cup\{0\}$. \\
$(1)$ We have
\begin{eqnarray*}
Par_N=Par_N^0\sqcup Par_N^1,
\end{eqnarray*}
where
\begin{eqnarray*}
Par_N^0&=&\{(-k,k);0\leq k\leq\lfloor\frac{N-1}{2}\rfloor\}\subset\mathbf{Z}_{N,\pm}\times\mathbf{Z}_{N,\pm}, \\
Par_N^1&=&\{(-k,k+1);0\leq k\leq\lfloor\frac{N-2}{2}\rfloor\}\subset\mathbf{Z}_{N,\pm}\times\mathbf{Z}_{N,\pm}.
\end{eqnarray*}
In particular there are $N$ parental pairs of zero for any $N$.
\end{prp}

\begin{proof}
(0) By Definition 5.5, if $rav_{\pm}(a,b)=0$, then
\begin{eqnarray*}
rav(\varphi_N^{-1}(a),\varphi_N^{-1}(b))=\varphi_N^{-1}(0)=0.
\end{eqnarray*}
This implies by Corollary 5.1 that $\varphi_N^{-1}(a)=0$ or $\frac{N}{2}<\varphi_N^{-1}(a)\leq N-1$, which is equivalent to the condition that $a\in\mathbf{Z}_{N,\leq 0}$. 

\noindent
(1) Suppose that $(a,b)\in Par_N$. Since we have $(0,0), (0,1)\in Par_N$ by Proposition 5.1, (0), we may assume that $a\neq 0$, and hence $a\in\mathbf{Z}_{N,-}$. The latter implies in particular that $\varphi_N^{-1}(a)\geq 2$. Therefore it follows from Proposition 5.3, (2) that
\begin{eqnarray*}
\varphi_N^{-1}(a)+\varphi_N^{-1}(b)\equiv 0, 1\pmod N.
\end{eqnarray*}
This is equivalent to
\begin{eqnarray}
a+b\equiv 0, 1\pmod N
\end{eqnarray}
by Remark 4.2. Since we can put $a=-k$ with $k\in [1,\lfloor\frac{N-1}{2}\rfloor]$, the condition (5.3) implies that $b=k$ or $b=k+1$. The latter alternative is excluded only when $N$ is odd and $k=\frac{N-1}{2}$ by the shape of $\mathbf{Z}_{N,\pm}$. The last assertion follows from the equality
\begin{eqnarray*}
|Par_N^0|+|Par_N^1|=(\lfloor\frac{N-1}{2}\rfloor+1)+(\lfloor\frac{N-2}{2}\rfloor+1)=N.
\end{eqnarray*}
This completes the proof.
 \end{proof}

For small $N$, the set $Par_N$ is given by the following:
\begin{eqnarray*}
\begin{array}{lcl}
N & \vline & Par_N \\
\hline
3 & \vline & (0,0), (0,1), (-1,1) \\
4 & \vline & (0,0), (0,1), (-1,1), (-1,2) \\
5 & \vline & (0,0), (0,1), (-1,1), (-1,2), (-2,2) \\
6 & \vline & (0,0), (0,1), (-1,1), (-1,2), (-2,2), (-2,3) \\
\end{array}
\end{eqnarray*}
\begin{center}
Table 5.1. $Par_N$ for $N=3,4,5, 6$
\end{center}

\begin{rem}
When $n=1$, there is only one ancestor of zero in $\mathbf{R}_{N,\pm}^1$, namely, the singleton $(0)$. This is because the discrete average map $Rav_{\pm}$ is defined to be the identity map on $\mathbf{R}_{N,\pm}^1$. For this reason we assume $n\geq 2$ from now on.
\end{rem}

Based on Proposition 5.2, we can understand completely the structure of the set of ancestors of zero, and hence the shape of the rhythm polynomial $Bav_N^0$. The following proposition is a rephrasing of Definition 5.4:

\begin{prp}
When $n\geq 2$, an element $\mathbf{a}=(a_0,\cdots,a_{n-1})\in\mathbf{R}_{N,\pm}^n$ is an ancestor of zero if and only if there exists an index $i\in\mathbf{Z}_n$ such that $(a_i,a_{i+_n1})\in Par_N$.
\end{prp}

\noindent
In view of this, we introduce the following:

\begin{df}
For any $(a,b)\in Par_N\setminus\{(0,0)\}$ and for any $n\in[2,N]$, we put
\begin{eqnarray*}
Anc_{(a,b)}^n=\{\mathbf{a}\in\mathbf{I}_{N,\pm}^n;(a_i,a_{i+_n1})=(a,b)\mbox{ for some }i\in\mathbf{Z}_n\},
\end{eqnarray*}
and
\begin{eqnarray}
Anc_{(a,b)}=\bigcup_{n=2}^NAnc_{(a,b)}^n.
\end{eqnarray}
When $(a,b)=(0,0)$, we put
\begin{eqnarray*}
Anc_{(0,0)}=\{(0)\}.
\end{eqnarray*}
Collecting these families of ancestors, we put
\begin{eqnarray*}
Anc_N=\bigcup_{(a,b)\in Par_N}Anc_{(a,b)},
\end{eqnarray*}
and call it {\rm the set of ancestors of zero}.
\end{df}

In order to formulate a Boolean counterpart of the notion of ancestor of zero, we put
\begin{eqnarray}
[a,b]_{N,\pm}&=&\varphi_N([\varphi_N^{-1}(a),\varphi_N^{-1}(b)]_N),\\
\mbox{$[$}a,b)_{N,\pm}&=&[a,b]_{N,\pm}\setminus\{b\}
\end{eqnarray}
for any $(a,b)\in\mathbf{Z}_{N,\pm}\times\mathbf{Z}_{N,\pm}$. Then we have the following:

\begin{prp}
Fix an arbitrary $(a,b)\in Par_N\setminus\{(0,0)\}$. For any $\mathbf{a}\in\mathbf{I}_{N,\pm}^n$ with $n\geq 2$, let $\mathbf{v}=ItoB(\mathbf{a})\in\mathbf{B}_{N,\pm}$. Then the following three conditions are equivalent:\\
$(1)$ $\mathbf{a}\in Anc_{(a,b)}$.\\
$(2)$ $\mathbf{a}\cap[a,b]_{N,\pm}=\{a,b\}$.\\
$(3)$ $v_a=1,v_{a+_N1}=0,\cdots,v_{b-_N1}=0,v_b=1$.\\
\end{prp}

\begin{rem}
Here in the item (3), we denote the addition on $\mathbf{Z}_{N,\pm}$ also by "$+_N$", which is used originally for the addition on $\mathbf{Z}_N$. Strictly speaking, for any $(a,b)\in\mathbf{Z}_{N,\pm}\times\mathbf{Z}_{N,\pm}$, we should define a new addition "$a+_{N,\pm}b$" by the rule 
\begin{eqnarray*}
a+_{N,\pm}b=\varphi_N(\varphi_N^{-1}(a)+_N\varphi_N^{-1}(b)).
\end{eqnarray*}
We, however, abuse the notation, since the context will tell us which meaning we are adopting.
\end{rem}

\begin{proof}
$(1)\Rightarrow (2)$: Let $\mathbf{a}=(a_0,\cdots, a_{n-1})$. The statement (1) implies by Definition 5.7 that there exists an index $i\in\mathbf{Z}_n$ such that $(a_i,a_{i+_n1})=(a,b)$. On the other hand, by (1.8), we have
\begin{eqnarray*}
\mathbf{Z}_{N,\pm}=\bigsqcup_{j\in\mathbf{Z}_n}[a_j,a_{j+_n1})_{N,\pm}.
\end{eqnarray*}
Since the right hand side is a disjoint union, we have $\mathbf{a}\cap [a,b]=\{a,b\}$.\\

\noindent
$(2)\Rightarrow (3)$: By the definition of the map $ItoB$, the $i$-th coordinate $v_i$, $i\in\mathbf{Z}_{N,\pm}$, of $\mathbf{v}=ItoB(\mathbf{a})$ is given by
\begin{eqnarray*}
v_i=
\left\{
\begin{array}{ll}
1, & i\in \mathbf{a},\\
0, & i\not\in \mathbf{a}. \\
\end{array}
\right.
\end{eqnarray*}
Therefore the equality $\mathbf{a}\cap[a,b]_{N,\pm}=\{a,b\}$ in (2) implies the equalities $v_a=1,v_{a+_N1}=0,\cdots,v_{b-_N1}=0,v_b=1$ in the statement (3). \\

\noindent
$(3)\Rightarrow (1)$: It follows from (3) and the definition of $ItoB$ that there exists an $i\in\mathbf{Z}_{N,\pm}$ such that $a_i=a, a_{i+_n1}=b$. Hence it follows from Definition 5.7 that the statement (1) holds true. This completes the proof.
\end{proof}

In view of this proposition, we introduce the notion of $\mathbf{B}$-ancestor of zero as follows:

\begin{df}
Every element in $ItoB(Anc_N)\subset \mathbf{B}_{N,\pm}$ are called a $\mathbf{B}$-ancestor of zero. We denote by $\mathbf{B}$-$Anc_N$ the set of $\mathbf{B}$-ancestors of zero. Furthermore for each $(a,b)\in Par_N$, we put
\begin{eqnarray*}
\mathbf{B}\mbox{-}Anc_{(a,b)}=ItoB(Anc_{(a,b)})
\end{eqnarray*}
so that we have
\begin{eqnarray*}
\mathbf{B}\mbox{-}Anc_N=\bigcup_{(a,b)\in Par_N}\mathbf{B}\mbox{-}Anc_{(a,b)}
\end{eqnarray*}
\end{df}

\noindent
The following notation will simplify our description below:

\begin{df}
For any $(a,b)\in Par_N\setminus\{(0,0)\}$, we denote by $\mathbf{e}_{[a,b]}\in\mathbf{F}_2^{b-a+1}$ the Boolean vector whose $i$-th coordinate $\mathbf{e}_{[a,b]}^i$ $(1\leq i\leq b-a+1)$ is defined by
\begin{eqnarray*}
\mathbf{e}_{[a,b]}^i=
\left\{
\begin{array}{ll}
1, & \mbox{ if }i=1\mbox{ or }i=b-a+1,\\
0, & \mbox{ if }1<i<b-a+1.\\
\end{array}
\right.
\end{eqnarray*}
In other words, the vector $\mathbf{e}_{[a,b]}$ is a Boolean vector of length $|[a,b]|$ such that the first and the last coordinates are $1$ and the others are $0$.
\end{df}

\noindent
The equivalence of (1) and (3) in Proposition 5.4 implies the following:

\begin{cor}
For any $(a,b)\in Par_N\setminus\{(0,0)\}$, let $pr_{[a,b]}:\mathbf{B}_{N,\pm}\rightarrow \mathbf{F}_2^{b-a+1}$ denote the projection defined by 
\begin{eqnarray*}
pr_{[a,b]}(v_{\ell},\cdots,v_g)=(v_a,\cdots,v_b).
\end{eqnarray*}
Then we have
\begin{eqnarray*}
\mathbf{B}\mbox{-}Anc_{(a,b)}=pr_{[a,b]}^{-1}(\mathbf{e}_{[a,b]}).
\end{eqnarray*}
\end{cor}

\noindent
This implies further the following:

\begin{cor}
For any $(a,b)\in Par_N\setminus\{(0,0)\}$, the number of elements in $Anc_{(a,b)}$ is given by
\begin{eqnarray*}
|Anc_{(a,b)}|=2^{N-1-(b-a)}.
\end{eqnarray*}
\end{cor}

\begin{proof}
By the bijectivity of the map $ItoB$, we have only to count the number of elements in $\mathbf{B}$-$Anc_{[a,b]}$, which coincides with $pr_{[a,b]}^{-1}(\mathbf{e}_{[a,b]})$ by Corollary 5.2. Hence we can compute as follows:
\begin{eqnarray*}
|Anc_{(a,b)}|&=&|\mathbf{B}\mbox{-}Anc_{(a,b)}|\\
&=&|pr_{[a,b]}^{-1}(\mathbf{e}_{[a,b]})|\\
&=&2^{N-(b-a+1)}.
\end{eqnarray*}
This finishes the proof.
\end{proof}

We can deduce from this corollary that the polynomial $Bav_N^0$ is {\it balanced}.

\begin{df}
A Boolean function $P:\mathbf{F}_2^n\rightarrow\mathbf{F}_2$ is said to be {\rm balanced} if
\begin{eqnarray*}
|P^{-1}(\{0\})|=|P^{-1}(\{1\})|=2^{n-1}.
\end{eqnarray*}
\end{df}

\noindent
For the importance of the balanced polynomials in cryptology, we refer the reader to [4, Chapter 3].

\begin{thm}
The $0$-th rhythm polynomial $Bav_N^0$ is balanced.
\end{thm}

\begin{proof}
The set of ancestors of zero is the disjoint union of $Anc_{(a,b)}$ $((a,b)\in Par_N\setminus\{(0,0)\})$ and $Anc_{(0,0)}$. Hence the total is computed through Corollary 5.3 as follows:
\begin{eqnarray}
&&\left(\sum_{(a,b)\in Par_N\setminus\{(0,0)\}}|Anc_{(a,b)}|\right)+|Anc_{(0,0)}|\nonumber\\
&&=\left(\sum_{(a,b)\in Par_N\setminus\{(0,0)\}}2^{N-1-(b-a)}\right)+1\nonumber\\
&&=\left(\sum_{(a,b)\in Par_N^0\setminus\{(0,0)\}}2^{N-1-(b-a)}+\sum_{(a,b)\in Par_N^1}2^{N-1-(b-a)}\right)+1\nonumber\\
&&=\left(\sum_{1\leq k\leq \lfloor\frac{N-1}{2}\rfloor}2^{N-1-(k-(-k))}+\sum_{0\leq k\leq \lfloor\frac{N-2}{2}\rfloor}2^{N-1-((k+1)-(-k))}\right)+1\nonumber\\
&&\hspace{70mm}(\mbox{by Proposition 5.2, (1))}\nonumber\\
&&=\left(\sum_{1\leq k\leq \lfloor\frac{N-1}{2}\rfloor}2^{N-1-2k}+\sum_{0\leq k\leq \lfloor\frac{N-2}{2}\rfloor}2^{N-1-(2k+1)}\right)+1
\end{eqnarray}
Here we notice that the set $\{2k;1\leq k\leq \lfloor\frac{N-1}{2}\rfloor\}$ coincides with the set of positive even numbers in the interval $[0,N-1]$, and that the set $\{2k+1;0\leq k\leq \lfloor\frac{N-2}{2}\rfloor\}$ coincides with the set of odd numbers in the interval $[0,N-1]$. Therefore we have
\begin{eqnarray*}
\{2k;1\leq k\leq \lfloor\frac{N-1}{2}\rfloor\}\cup \{2k+1;0\leq k\leq \lfloor\frac{N-2}{2}\rfloor\}=[1,N-1]
\end{eqnarray*}
Hence the rightmost side of (5.7) is equal to
\begin{eqnarray*}
\left(\sum_{1\leq j\leq N-1}2^{N-1-j}\right)+1=\left(\sum_{0\leq j\leq N-2}2^j\right)+1=2^{N-1}.
\end{eqnarray*}
This completes the proof.
\end{proof}

\subsection{Main theorem on the structure of the rhythm polynomial $Bav_N^0$}
In order to state our main theorem on $Bav_N^0$, we introduce some notations:

\begin{df}
For any $(a,b)\in Par_N\setminus\{(0,0)\}$, we denote by $f_{[a,b]}$ the Boolean polynomial given by
\begin{eqnarray*}
f_{[a,b]}=(y_a+1)\left(\prod_{k=a+1}^{b-1}y_k\right)(y_b+1).
\end{eqnarray*}
Furthermore we set
\begin{eqnarray*}
f_{[0,0]}^N=(y_0+1)\prod_{k\in\mathbf{Z}_{N,\pm}\setminus \{0\}}y_k.
\end{eqnarray*}
\end{df}

\begin{rem}
Notice that the definition of $f_{[a,b]}$ for any $[a,b]\in Par_N\setminus\{(0,0)\}$ does \underline{not} depend on $N$, but only on the values of $a,b$. On the other hand the definition of $f_{[0,0]}^N$ depends on $N$. This is why we put "$N$" on the superscript.
\end{rem}

The general form of the Boolean polynomial $Bav_N^0$ is given by the following:

\begin{thm}
For any $N\geq 3$, we have
\begin{eqnarray}
Bav_N^0=\sum_{(a,b)\in Par_N\setminus\{(0,0)\}}f_{[a,b]}+f_{[0,0]}^N.
\end{eqnarray}
\end{thm}

\noindent
Our proof of this theorem will be given later in the subsection 5.6, after we prepare several results concerning the algebraic standard forms of special Boolean functions.

\subsection{Preliminaries for the proof of Theorem 5.2}
For any $S=\{s_1,\cdots,s_k\}\subset\mathbf{Z}_N$, we denote by $pr_S:\mathbf{F}_2^N\rightarrow\mathbf{F}_2^{|S|}$ the projection defined by 
\begin{eqnarray*}
pr_S(x_0,\cdots,x_{N-1})=(x_{s_1},\cdots,x_{s_k})
\end{eqnarray*}
for any $(x_0,\cdots,x_{N-1})\in\mathbf{F}_2^N$.

\begin{prp}
Let $P=P(x_0,\cdots,x_{N-1}):\mathbf{F}_2^N\rightarrow \mathbf{F}_2$ be a Boolean function with the following property: There exists a subset $S=\{s_1,\cdots,s_k\}\subset \mathbf{Z}_N$ such that
\begin{eqnarray*}
P^{-1}(\{1\})=pr_S^{-1}(\{(b_1,\cdots,b_k)\})
\end{eqnarray*}
for some $(b_1,\cdots,b_k)\in\mathbf{F}_2^k$. Then we have
\begin{eqnarray*}
P(x_0,\cdots,x_{N-1})=x_{s_1}^{(b_1)}x_{s_2}^{(b_2)}\cdots x_{s_k}^{(b_k)},
\end{eqnarray*}
where we put
\begin{eqnarray*}
x_i^{(0)}&=&x_i+1,\\
x_i^{(1)}&=&x_i,
\end{eqnarray*}
for any Boolean variable $x_i, i\in\mathbf{Z}_N$.
\end{prp}

\begin{proof}
Note that
\begin{eqnarray*}
&&pr_S^{-1}(\{(b_1,\cdots,b_k)\})\\
&&=\{(a_0,\cdots,a_{N-1})\in\mathbf{F}_2^N;(a_{s_1},\cdots,a_{s_k})=(b_1,\cdots,b_k)\}.
\end{eqnarray*}
Hence, denoting the complement $\mathbf{Z}_N\setminus S$ by $T=\{j_1,\cdots,j_{N-k}\}$, we see that
\begin{eqnarray*}
P(x_0,\cdots,x_{N-1})=(x_{j_1}+\overline{x_{j_1}})\cdots(x_{j_{N-k}}+\overline{x_{j_{N-k}}})x_{s_1}^{(b_1)}x_{s_2}^{(b_2)}\cdots x_{s_k}^{(b_k)}
\end{eqnarray*}
by the algorithm explained in the subsection 4.1. Since $x_{j_t}+\overline{x_{j_t}}=1$ holds for any $t\in[1,N-k]$, we finish the proof.
\end{proof}

We also need an equivalent form of the above proposition, where we change the set of indices $\mathbf{Z}_N$ to $\mathbf{Z}_{N,\pm}$. Furthermore, since we have observed in the Table 4.5 that the Boolean variable "$y=x+1$" is superior to the original $x$ for our study, we employ the negated variables as our basic building blocks. The following proposition is a direct consequence of Proposition 5.5:

\begin{prp}
Let $P=P(y_{\ell},\cdots,y_g):\mathbf{F}_2^N\rightarrow \mathbf{F}_2$ be a Boolean function with the following property: There exists a subset $S=\{s_1,\cdots,s_k\}\subset \mathbf{Z}_{N,\pm}$ such that
\begin{eqnarray*}
P^{-1}(\{1\})=pr_S^{-1}(\{(b_1,\cdots,b_k)\})
\end{eqnarray*}
holds for some$(b_1,\cdots,b_k)\in\mathbf{F}_2^k$. Then we have
\begin{eqnarray*}
P(y_{\ell},\cdots,y_g)=y_{s_1}^{(b_1)}y_{s_2}^{(b_2)}\cdots y_{s_k}^{(b_k)},
\end{eqnarray*}
where we put
\begin{eqnarray*}
y_i^{(0)}&=&y_i,\\
y_i^{(1)}&=&y_i+1,
\end{eqnarray*}
for any Boolean variable $y_i, i\in\mathbf{Z}_{N,\pm}$.
\end{prp}

\subsection{Ancestors of zero when $N=6$}
Before we give a proof of Theorem 5.2, we examine the case when $N=6$. This might help the reader to understand our proof for general cases. \\

\newpage
\noindent
The set $Par_6$ of parental pairs of zero has been found in Proposition 5.2 as follows:
\begin{eqnarray*}
Par_6=\{(0,0),(0,1),(-1,1),(-1,2),(-2,2),(-2,3)\}.
\end{eqnarray*}
The figure below illustrates the ancestors of zero when $N=6$:

\noindent
\begin{center}
\begin{figure}[h]
    \begin{tabular}{ccc}
       \begin{minipage}[t]{0.2\hsize}
        \centering
        \includegraphics[scale=0.3]{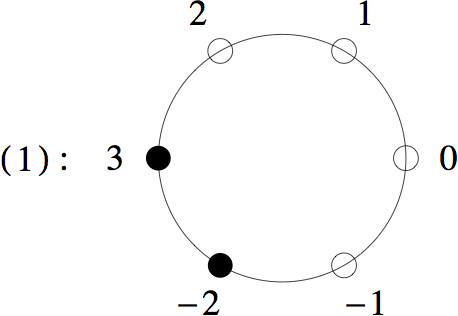}

       \end{minipage} &
       \begin{minipage}[t]{0.2\hsize}
        \centering
        \includegraphics[scale=0.3]{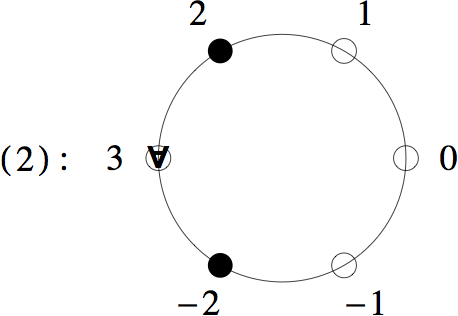}
        \end{minipage} &      
      \begin{minipage}[t]{0.2\hsize}
        \centering
        \includegraphics[scale=0.3]{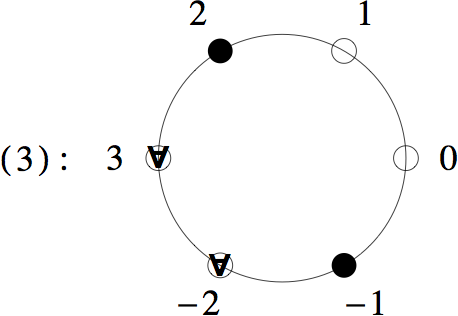}
        \end{minipage}

     \end{tabular}
  \end{figure}
  
  \begin{figure}[h]
    \begin{tabular}{ccc}
      \begin{minipage}[t]{0.2\hsize}
        \centering
        \includegraphics[scale=0.3]{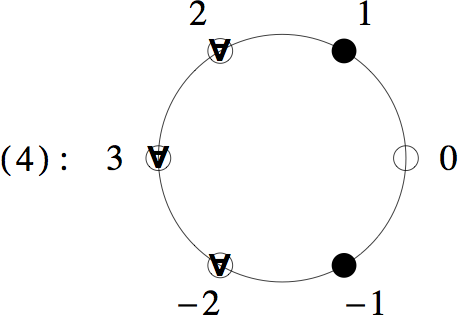}

       \end{minipage} &
      \begin{minipage}[t]{0.2\hsize}
        \centering
        \includegraphics[scale=0.3]{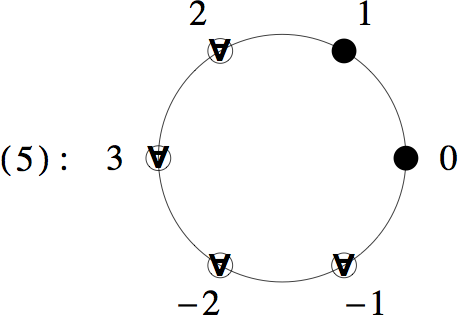}
      \end{minipage} &      
      \begin{minipage}[t]{0.2\hsize}
        \centering
        \includegraphics[scale=0.3]{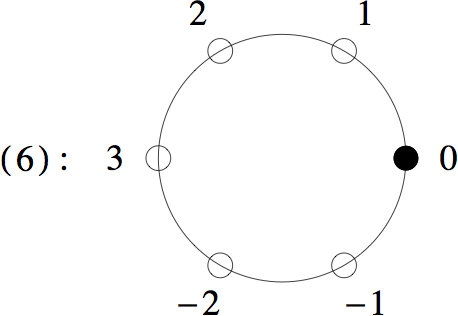}
        \label{ラベル3}
      \end{minipage}

    \end{tabular}
  \end{figure}
\end{center}
\hspace{20mm}Fig. 5.1. Ancestors of zero for $N=6$\\\\

\noindent
Here are six families (1)-(6) of ancestors of 0, which are determined in Proposition 5.2 and Corollary 5.2. Each of six large circles represents the unit circle on the complex plane, and each member $k$ of $\mathbf{Z}_{6,\pm}$ are placed at the point $\exp(k\pi i/3)$. The meanings of symbols in the figure are as follows:\\

\noindent
(1) The numbers surrounding the unit circle specify the names of the elements in $\mathbf{Z}_{6,\pm}$.\\
(2) Black dots represent the parental pairs of zero found in Proposition 5.2. \\
(3) Each of small blank circle signifies that the corresponding element does not appear in the ancestors. \\
(4) Each of circles having "$\forall$" in its interior means that the corresponding element may or may not be present in the ancestors. \\

\noindent
We should be careful about the rather exceptional case when $n=1$, which is depicted as "(6)". Thus the $\mathbf{B}$-ancestors of zero which correspond to the six figures are given by the following list:
\begin{eqnarray*}
\begin{array}{lll}
\mbox{name} & \mbox{$\mathbf{B}$-ancestors of zero in }\mathbf{B}_{6,\pm} & \mbox{number of elements} \\
\hline
(1) & \{e_{[-2,3]}\} & 1 \\
(2) & pr_{[-2,2]}^{-1}(\{e_{[-2,2]}\}) & 2\\
(3) &  pr_{[-1,2]}^{-1}(\{e_{[-1,2]}\}) & 4 \\
(4) &  pr_{[-1,1]}^{-1}(\{e_{[-1,1]}\}) & 8 \\
(5) &  pr_{[0,1]}^{-1}(\{e_{[0,1]}\}) & 16 \\
(6) & \{(0,0,1,0,0,0)\} & 1\\
Total & & 32(=2^{6-1}) \\
\end{array}
\end{eqnarray*}
\begin{center}
Table 5.2. Six families of the $\mathbf{B}$-ancestors of zero for $N=6$
\end{center}

\noindent
Let $S_i$ denote the subset of $\mathbf{B}_{6,\pm}$ which is specified in the $i$-th row $(1\leq i\leq 6)$, and let $f_{S_i}$ be the Boolean function which takes the value 1 (resp. 0) at every element of $S_i$ (resp. $\mathbf{B}_{6,\pm}\setminus S_i)$. Then it follows from Proposition 5.6 and Definition 5.11 that
\begin{eqnarray*}
f_{S_1}&=&f_{[-2,3]}=(y_{-2}+1)y_{-1}y_0y_1y_2(y_3+1), \\
f_{S_2}&=&f_{[-2,2]}=(y_{-2}+1)y_{-1}y_0y_1(y_2+1), \\
f_{S_3}&=&f_{[-1,2]}=(y_{-1}+1)y_0y_1(y_2+1), \\
f_{S_4}&=&f_{[-1,1]}=(y_{-1}+1)y_0(y_1+1), \\
f_{S_5}&=&f_{[0,1]}=(y_0+1)(y_1+1), \\
f_{S_6}&=&f_{[0,0]}^6=(y_0+1)y_{-2}y_{-1}y_1y_2y_3.
\end{eqnarray*}
By summing these up, we have
\begin{eqnarray*}
Bav_0^6=\sum_{(a,b)\in Par_6\setminus\{(0,0)\}}f_{[a,b]}+f_{[0,0]}^6,
\end{eqnarray*}
which shows the validity of (5.6) of Theorem 5.2 when $N=6$. Furthermore the total of the numbers of ancestors in the lines (1)-(6) is equal to 32, which is exactly the half of the number 64 of Boolean vectors in $\mathbf{B}_{6,\pm}$. Therefore $Bav_6^0$ is balanced as is assured by Theorem 5.1.

\subsection{Proof of Theorem 5.2}
For any subset $U\subset\mathbf{B}_{N,\pm}$, we denote by $P_U$ the polynomial function on $\mathbf{B}_{N,\pm}$ such that
\begin{eqnarray*}
P_U^{-1}(\{1\})=U.
\end{eqnarray*}
Namely $P_U$ takes the value 1 on $U$ and the value 0 outside $U$. Recall that we have $\mathbf{B}$-$Anc_N(a,b)=pr_{[a,b]}^{-1}(\{e_{[a,b]}\})$ for any $(a,b)\in Par_N\setminus\{(0,0)\}$ by Corollary 5.2. This implies by Proposition 5.6 and definition 5.11 that
\begin{eqnarray*}
P_{\mathbf{B}-Anc_N(a,b)}=f_{[a,b]}.
\end{eqnarray*}
The remaining parental pair $(0,0)\in Par_N$ corresponds to the one-element rhythm $(0)\in\mathbf{R}_N^1$ (see Remark 5.1). Since $RtoB((0))=(v_i)_{i\in\mathbf{Z}_{N,\pm}}$, where
\begin{eqnarray*}
v_i=
\left\{
\begin{array}{ll}
1, & \mbox{ if }i=0,\\
0, & \mbox{ if }i\neq 0,\\
\end{array}
\right.
\end{eqnarray*}
we have
\begin{eqnarray*}
P_{\mathbf{B}-Anc_N(0,0)}=f_{[0,0]}.
\end{eqnarray*}
Thus we obtain the equality (5.8). This completes the proof of Theorem 5.2.\\

\noindent
We check the validity of Theorem 5.2 by our computational results in Subsection 4.2.\\

\noindent
Example 5.1. The case when $N=3$: The equality (5.8) in Theorem 5.2 asserts that
\begin{eqnarray*}
Bav_3^0=\sum_{(a,b)\in Par_3\setminus\{(0,0)\}}f_{[a,b]}+f_{[0,0]}^3.
\end{eqnarray*}
Recall that
\begin{eqnarray*}
Par_3=\{(0,0),(0,1),(-1,1)\}
\end{eqnarray*}
by Proposition 5.2, and the corresponding polynomials $f_{[a,b]}$, $(a,b)\in Par_3\setminus \{(0,0)\}$ and $f_{[0,0]}^3$ are defined in Definition 5.11 as follows:
\begin{eqnarray*}
f_{[0,1]}&=&(y_0+1)(y_1+1),\\
f_{[-1,1]}&=&(y_{-1}+1)y_0(y_1+1),\\
f_{[0,0]}^3&=&(y_0+1)y_{-1}y_1.
\end{eqnarray*}
Hence we have
\begin{eqnarray*}
Bav_3^0&=&(y_0+1)(y_1+1)+(y_{-1}+1)y_0(y_1+1)+(y_0+1)y_{-1}y_1\\
&=&(1+\underline{y_0}+y_1+\underline{y_0y_1})+y_0(\underline{1}+y_{-1}+\underline{y_1}+\underline{y_{-1}y_1})+y_{-1}(\underline{y_0}+1)y_1\\
&=&1+y_1+y_{-1}y_0+y_{-1}y_1,\\
&&\hspace{30mm}(\mbox{by cancelling the underlined terms})
\end{eqnarray*}
and we see that the rightmost side coincides with Table 4.5 for $N=3$.\\

\noindent
Example 5.2. The case when $N=4$: By Theorem 5.2, Proposition 5.2, and Definition 5.11, we can compute as follows:
\begin{eqnarray*}
Bav_4^0&=&\sum_{(a,b)\in Par_4\setminus\{(0,0)\}}f_{[a,b]}+f_{[0,0]}^4 \\
&=&f_{[0,1]}+f_{[-1,1]}+f_{[-1,2]}+f_{[0,0]}^4\\
&=&(y_0+1)(y_1+1)+(y_{-1}+1)y_0(y_1+1)+(y_{-1}+1)y_0y_1(y_2+1)\\
&&\hspace{5mm}+(y_0+1)y_{-1}y_1y_2\\
&=&(1+\underline{y_0}+y_1+\underline{y_0y_1})+y_0(\underline{1}+y_{-1}+\underline{y_1}+\underline{y_{-1}y_1})\\
&&+y_0y_1(1+\underline{y_{-1}}+y_2+\underline{y_{-1}y_2}))+(\underline{y_0}+1)y_{-1}y_1y_2\\
&=&1+y_1+y_{-1}y_0+y_0y_1+y_0y_1y_2+y_{-1}y_1y_2.\\
&&\hspace{30mm}(\mbox{by cancelling the underlined terms})
\end{eqnarray*}
The rightmost side coincides with Table 4.5 for $N=4$.

\subsection{Recurrence formula}
Inspecting the shapes of $Bav_N^0$  found in Theorem 5.2, we are led naturally to the following:

\begin{thm}
When $N$ is even and $N=2m$, we have
\begin{eqnarray}
Bav_{2m}^0=Bav_{2m-1}^0+y_{-m+2}\cdots y_{-1}y_1\cdots y_{m-1}(y_m+1)(y_0+y_{-m+1}).
\end{eqnarray}
When $N$ is odd and $N=2m+1$, we have
\begin{eqnarray}
Bav_{2m+1}^0=Bav_{2m}^0+y_{-m+1}\cdots y_{-1}y_1\cdots y_{m-1}(y_{-m}+1)(y_0+y_{m}).
\end{eqnarray}
\end{thm}

\begin{proof}
It follows from Proposition 5.2 that
\begin{eqnarray*}
Par_N\subset Par_{N+1}
\end{eqnarray*}
holds for any $N\geq 3$. Hence we have only to consider their difference. We divide the proof into two parts according to the parity of $N$.\\

\noindent
1) The case when $N=2m\geq 4$. It follows from Proposition 5.2 that
\begin{eqnarray*}
Par_{2m}\setminus Par_{2m-1}=\{[-m+1,m]\}.
\end{eqnarray*}
Before we start our computation, for the convenience of the reader, we recall what members are in $\mathbf{Z}_{2m,\pm}$ or in $\mathbf{Z}_{2m-1,\pm}$:
\begin{eqnarray*}
\mathbf{Z}_{2m,\pm}&=&\{x\in\mathbf{Z};-m+1\leq x\leq m\},\\
\mathbf{Z}_{2m-1,\pm}&=&\{x\in\mathbf{Z};-m+1\leq x\leq m-1\}.
\end{eqnarray*}
Keeping in mind the fact that $-1=1$ in $\mathbf{F}_2$, we can compute as follows:
\begin{eqnarray*}
&&Bav_{2m}^0-Bav_{2m-1}^0\\
&=&f_{[-m+1,m]}+f_{[0,0]}^{2m}-f_{[0,0]}^{2m-1}\\
&=&(y_{-m+1}+1)\left(\prod_{k=-m+2}^{m-1}y_k\right)(y_m+1)\\
&&+\left((y_0+1)\prod_{k\in\mathbf{Z}_{2m,\pm}\setminus \{0\}}y_k-(y_0+1)\prod_{k\in\mathbf{Z}_{2m-1,\pm}\setminus \{0\}}y_k\right)\\
&=&(y_{-m+1}+1)\left(\prod_{k=-m+2}^{m-1}y_k\right)(y_m+1)\\
&&\hspace{5mm}+(y_0+1)\left(\prod_{k\in\mathbf{Z}_{2m-1,\pm}\setminus \{0\}}y_k\right)(y_{m}-1)\\
&=&y_{-m+2}\cdots y_{-1}y_1\cdots y_{m-1}\\
&&\times\left((y_{-m+1}+1)y_0(y_m+1)+y_{-m+1}(y_0+1)(y_m+1)\right)\\
&=&y_{-m+2}\cdots y_{-1}y_1\cdots y_{m-1}(y_m+1)\\
&&\times\left((y_{-m+1}+1)y_0+y_{-m+1}(y_0+1)\right)\\
&=&y_{-m+2}\cdots y_{-1}y_1\cdots y_{m-1}(y_m+1)(y_0+y_{-m+1}).
\end{eqnarray*}
This shows the validity of (5.9).\\

\noindent
2) The case when $N=2m+1\geq 5$. It follows Proposition 5.2 that
\begin{eqnarray*}
Par_{2m+1}\setminus Par_{2m}=\{[-m,m]\}.
\end{eqnarray*}
This time, the members of $\mathbf{Z}_{2m+1,\pm}$ and $\mathbf{Z}_{2m,\pm}$ are given by
\begin{eqnarray*}
\mathbf{Z}_{2m+1,\pm}&=&\{x\in\mathbf{Z};-m\leq x\leq m\},\\
\mathbf{Z}_{2m,\pm}&=&\{x\in\mathbf{Z};-m+1\leq x\leq m\}.
\end{eqnarray*}
Hence we can compute as follows:
\begin{eqnarray*}
&&Bav_{2m+1}^0-Bav_{2m}^0\\
&=&f_{[-m,m]}+f_{[0,0]}^{2m+1}-f_{[0,0]}^{2m}\\
&=&(y_{-m}+1)\left(\prod_{k=-m+1}^{m-1}y_k\right)(y_m+1)\\
&&+\left((y_0+1)\prod_{k\in\mathbf{Z}_{2m+1,\pm}\setminus \{0\}}y_k-(y_0+1)\prod_{k\in\mathbf{Z}_{2m,\pm}\setminus \{0\}}y_k\right)\\
&=&(y_{-m}+1)\left(\prod_{k=-m+1}^{m-1}y_k\right)(y_m+1)\\
&&+(y_0+1)\left(\prod_{k\in\mathbf{Z}_{2m,\pm}\setminus \{0\}}y_k\right)(y_{-m}-1)\\
&=&y_{-m+1}\cdots y_{-1}y_1\cdots y_{m-1}\\
&&\times\left((y_{-m}+1)y_0(y_m+1)+(y_{-m}+1)(y_0+1)y_m\right)\\
&=&y_{-m+1}\cdots y_{-1}y_1\cdots y_{m-1}(y_{-m}+1)\\
&&\times\left(y_0(y_m+1)+(y_0+1)y_m\right)\\
&=&y_{-m+1}\cdots y_{-1}y_1\cdots y_{m-1}(y_{-m}+1)(y_0+y_m).
\end{eqnarray*}
This shows the validity of (5.10). This completes the proof.
\end{proof}

We have displayed in Table 4.5 the rhythm polynomials $Bav_N^0$ for $N=3,\cdots,6$. We reexamine the results through Theorem 5.3.\\

\noindent
Example 5.3. The case when $N=4$. It follows from (5.9) with $m=2$ that
\begin{eqnarray*}
Bav_4^0&=&Bav_3^0+y_1(y_2+1)(y_0+y_{-1})\\
&=&(1+y_1+y_{-1}y_0+\underline{y_{-1}y_1})+y_1(y_{-1}y_2+y_0y_2+\underline{y_{-1}}+y_0)\\
&=&1+y_1+y_{-1}y_0+y_0y_1+y_{-1}y_1y_2+y_0y_1y_2,\\
&&\hspace{30mm}(\mbox{by canceling the underlined terms})
\end{eqnarray*}
where the last polynomial coincides with the result in Table 4.5 for $N=4$.\\

\noindent
Example 5.4. The case when $N=5$. It follows from (5.10) with $m=2$ that
\begin{eqnarray*}
Bav_5^0&=&Bav_4^0+y_{-1}y_1(y_{-2}+1)(y_0+y_2)\\
&=&(1+y_1+y_{-1}y_0+y_0y_1+\underline{y_{-1}y_1y_2}+y_0y_1y_2)\\
&&\hspace{20mm}+y_{-1}y_1(y_{-2}y_0+y_{-2}y_2+y_0+\underline{y_2})\\
&=&1+y_1+y_{-1}y_0+y_0y_1+y_{-1}y_0y_1+y_0y_1y_2\\
&&\hspace{20mm}+y_{-2}y_{-1}y_0y_1+y_{-2}y_{-1}y_1y_2,\\
&&\hspace{30mm}(\mbox{by canceling the underlined terms})
\end{eqnarray*}
which coincides with the result in Table 4.5 for $N=5$.\\\\

\noindent
{\bf References}\\
\noindent
$[1]$ Cusick, T. W., St\u{a}nic\u{a}, {\it Cryptographic Boolean Functions and Applications}. Elsevier/Academic Press, Amsterdam, the Netherlands, 2009.\\
$[2]$ Hazama, F., Iterative method of construction for amooth rhythms.  Journal of Mathematics and Music(2021), On-Line-First, 1-20,\\
https://doi.org/10.1080/17459737.2021.1924303\\
$[3]$ MacWilliams, F. J., Sloane, N. J. A., {\it The Theory of Error-Correcting Codes}. Elsevier/NorthHolland, Amsterdam, the Netherlands, 1977. \\
$[4]$ Whitesitt, J. E., {\it Boolean Algebra and Its Application}, Dover Bookds on Computer science, 2010.\\

\end{document}